\documentclass[12pt]{article}

\usepackage{amsmath,amssymb,latexsym, amsbsy, amsfonts, amscd, amsthm, mathrsfs, xy,comment, xcolor}
\usepackage{multirow}
\usepackage{hyperref}
\usepackage{enumitem, empheq, url}
 \usepackage[bibtex-style]{amsrefs}
\xyoption{all}
\usepackage{tikz-cd}
\usepackage{rotating}

\theoremstyle{plain}

\newtheorem{theorem}{Theorem}

\newtheorem{prop}[theorem]{Proposition}
\newtheorem{corollary}[theorem]{Corollary}

\newtheorem{lemma}[theorem]{Lemma}

\theoremstyle{definition}
\newtheorem{definition}[theorem]{Definition}
\newtheorem{remark}[theorem]{Remark}

\long\def\symbolfootnote[#1]#2{\begingroup
\def\thefootnote{\fnsymbol{footnote}}\footnote[#1]{#2}\endgroup}

\def\lra{\longrightarrow}

\def\A{\mathbf{A}}

\def\N{\mathrm{N}}
\def\1{\mf{1}}

\DeclareMathOperator{\id}{id}

\DeclareMathOperator{\Fitt}{Fitt}

\DeclareMathOperator{\Ind}{Ind}

\DeclareMathOperator{\coker}{coker} 
\DeclareMathOperator{\ord}{ord}

 \DeclareMathOperator{\Cl}{Cl}
 \DeclareMathOperator{\red}{red}

\DeclareMathOperator{\im}{im}

\newcommand{\mat}[4]{\begin{pmatrix}{#1} & {#2} \\ {#3} & {#4}
\end{pmatrix}}

\newcommand{\mf}{\mathfrak }

\def\fp{\mathfrak{p}}

\def\fP{\mathfrak{P}}

\def\fO{\mathfrak{O}}

\def\Z{\mathbf{Z}}
\def\F{\mathbf{F}}
\def\Q{\mathbf{Q}}

\def\G{\mathbf{G}}
\def\R{\mathbf{R}}

\def\bdf{\begin{defn}}
\def\edf{\end{defn}}

\def\cO{\mathcal{O}}

\def\Gal{{\rm Gal}}

\def\ra{\rightarrow}

\def\ab{{\rm ab}}

\def\ram{\text{ram}}
\def\ab{\text{ab}}

\tikzset{
  symbol/.style={
    draw=none,
    every to/.append style={
      edge node={node [sloped, allow upside down, auto=false]{$#1$}}}
  }
}

\begin{document}
\baselineskip 15.8pt

\makeatletter
\renewcommand{\maketitle}{\bgroup\setlength{\parindent}{0pt}
\begin{center}
  \LARGE{\textbf{\@title}}
  
  \vspace{4mm}
  
  \large{\textsc{\@author}}
  
  \vspace{4mm}
\end{center}\egroup
}
\makeatother

\title{On the minus part of the equivariant Tamagawa number conjecture for an abelian CM extension of a totally real field}
\author{Samit Dasgupta, Mahesh Kakde, and Jesse Silliman}

\maketitle

{\small 
\paragraph{Abstract.}
Let $F$ be a totally real field and $K$ a finite abelian CM extension of $F$.  Using class field theory, we show that our previous results giving a strong form of the Brumer--Stark conjecture imply the minus part of the equivariant Tamagawa number conjecture for the Tate motive associated to $K/F$.  We work integrally over $\Z$, in particular the prime 2 is not inverted.  This note can be viewed as our perspective on the recent work of Bullach, Burns, Daoud, and Seo.  Following their philosophy, we show that the functorial properties (i.e.~norm compatibilities) connecting the strong Brumer--Stark conjecture for varying number fields actually imply the minus part of the equivariant Tamagawa number conjecture.
}

 \tableofcontents

 \section{Introduction}

Let $F$ be a totally real field and $K$ a finite abelian CM extension of $F$.  Write \[ G_K = \Gal(K/F).  \] In a previous paper with Jiuya Wang \cite{bsapet}, we proved refined versions of the Brumer--Stark conjecture for $K/F$.  In this note, we consider the minus part of the equivariant Tamagawa number conjecture for the Tate motive attached to $K/F$, and call this ``the $\textrm{ETNC}^{-}$.''  In a recent preprint \cite{bbds}, Bullach, Burns, Daoud, and Seo develop the theory of \emph{scarcity of Euler systems} and as a consequence prove that our results on the strong Brumer--Stark conjecture proved in \cite{dk} actually imply the $\textrm{ETNC}^{-}$ over $\Z[1/2]$.  The purpose of this note is to give our perspective on this implication yielding a concise proof of the $\textrm{ETNC}^{-}$.  We work integrally over $\Z$ rather than $\Z[1/2]$.  While we do not directly apply any results or arguments of \cite{bbds}, our method is directly inspired by theirs---by passing to a suitable inverse limit, the discrepancy between strong Brumer--Stark and the $\textrm{ETNC}^{-}$ becomes a non-zerodivisor and hence can be inverted.

We begin by stating the $\textrm{ETNC}^{-}$, following Kurihara's formulation~\cite{kurihara}.
 First we must set some notation. 
 Let $c \in G_K$ denote the unique complex conjugation and define
 \[ \Z[G_K]_- = \Z[G_K]/(c+1). \]
 For any $G_K$-module $M$, we write \[M_- = M \otimes_{\Z[G_K]} \Z[G_K]_- = M/(c+1)M. \]
 Let $S_\infty$ denote the set of infinite primes of $F$, and let $S_{\ram}(K/F)$ denote the set of finite primes of $F$ ramified in $K$. 
Let \begin{equation} \label{e:sdef}
 S = S(K) = S_\infty \cup S_{\ram}(K/F). 
 \end{equation}
 Let $T$ denote a finite set of primes of $F$, disjoint from $S(K)$, satisfying the usual condition of Deligne--Ribet, namely that $T$ contains two primes of different residue characteristic, or one prime of residue characteristic larger than $[F:\Q] + 1$.
Let  $\nabla_{S}^{T}(K)$ denote the Ritter--Weiss $G_K$-module defined in \cite{dk}*{Appendix A}.  This is a transpose of the Selmer group denoted
$\mathcal{S}_{K, S}^T$ in \cite{kurihara} and $\mathcal{S}_{S,T}(\mathbb{G}_{m/K})$ in \cite{bks}.

 The module $\nabla_{S}^{T}(K)$ is the cokernel of a map
\begin{equation} \label{e:fdef}
 f_{S}^{T}(K) \colon V_{S'}^{T}(K) \longrightarrow B_{S'}(K), \end{equation}
where $S'$ is an auxiliary set of places containing $S$ satisfying certain conditions.
Here $B_{S'}(K)$ is a free $\Z[G_K]$-module of rank $(\#S' - 1)$, and $V_{S'}^{T}(K)$ is a projective 
$\Z[G_K]$-module of the same (constant) rank.  All of these objects are defined in \S\ref{s:rwm}.
Kurihara has conjectured that $V_{S'}^{T}(K)$ is in fact free over $\Z[G_K]$;  in this paper we work entirely on the minus side, and we prove that $V_{S'}^{T}(K)_-$ is free over $\Z[G_K]_-$.

Write $f_K = f_{S}^{T}(K)$ and let $f_{K, -}$ denote the map obtained from $f_K$ after projecting to the minus side (i.e.\ tensoring with
 $\Z[G_K]_-$ over $\Z[G_K]$).
For any CM subfield $E \subset K$ containing $F$, the $G_E$-module $V_{S'}^{T}(E)$ may be canonically identified with the space of $\Gal(K/E)$-co-invariants of $V_{S'}^{T}(K)$, and similarly for the  modules denoted $B$.  Therefore, a choice of bases for the domain and codomain of $f_{K,-}$  naturally induces bases for the domain and codomain of $f_{E,-}  = f_{S(E)}^{T}(E)_-$.  The formulation of $\textrm{ETNC}^{-}$ that we prove in this paper is the following.

\begin{theorem}  \label{t:ETNC}  The module $V_{S'}^{T}(K)_-$ is free of rank $(\#S' - 1)$ over $\Z[G_K]_-$.
Fix bases for the domain and codomain of $f_{K,-} = f_{S}^{T}(K)_-$, thereby inducing bases for the domain and codomain of $f_{E,-} = f_{S(E)}^{T}(E)_-$ for all CM fields $E\subset K$ containing $F$.  
There exists a unit $y_K \in (\Z[G_K]_-)^\times$ such that
\begin{equation} \label{e:ETNC}
 \det(f_{E,-}) = \red_{K/E}({y}_K) \cdot \Theta_{S(E),T}(E/F) \text{ in } \Z[G_E]_- \end{equation}
for all $E$, where $\red_{K/E}({y}_K)$ denotes the canonical reduction of $y_K$ in  $(\Z[G_E]_-)^\times$.  
\end{theorem}
Here $ \Theta_{S(E),T}(E/F)$ is the usual Stickelberger element in $\Z[G_E]_-$ (for details, see \cite{bsapet}*{\S1}) determined by 
\[
\chi(\Theta_{S(E), T}(E/F)) = L_{S(E), T}(\chi^{-1}, 0),
\]
for all totally odd characters $\chi$ of $G_E$. Further, $L_{S(E), T}(\chi^{-1}, 0)$ is the value at $s=0$ of the analytic continuation of the function 
\[
L_{S(E), T}(\chi^{-1}, s) = \prod_{\mathfrak{p} \in T} (1 - \chi^{-1}(\mathfrak{p}) \N\mathfrak{p}^{1-s}) \prod_{\mathfrak{p} \notin S(E)} \frac{1}{1 - \chi^{-1}(\mathfrak{p}) \N\mathfrak{p}^{-s}},  \qquad \text{Re}(s) > 1.
\]
As explained in \cite{tate}*{Proposition~I.3.4}, we have
\begin{equation} \label{e:tate}
 \ord_{s=0} L_{S(E), T}(\chi^{-1},s) = \begin{cases} | \{v \in S(E) \colon \chi(G_v) = 1 \}| & \chi \neq 1 \\
|S(E)| - 1 & \chi=1.
\end{cases} \end{equation}

\subsection{Consequences} \label{s:icon}

Before sketching our proof of Theorem~\ref{t:ETNC}, we connect our formulation to others and describe some important consequences.

\begin{itemize}
\item The statement of Theorem~\ref{t:ETNC} is clearly equivalent to the projection to the minus side of Kurihara's conjecture given in \cite{kurihara}*{Conjecture 3.4}.
\item In \cite{kurihara}*{Proposition 3.5}, Kurihara proves that the minus part of his statement is equivalent to the minus part of the ``Leading Term Conjecture'' (LTC) stated in \cite{bks}. The statement of the LTC and a discussion of its relationship to Theorem~\ref{t:ETNC} are given in \S\ref{s:ltc}.
\item Burns, Kurihara, and Sano prove that their LTC is equivalent
 to the original form of the ETNC stated by Burns and Flach in
 \cite{bf}; in particular, the minus part of the LTC is equivalent
 to $\textrm{ETNC}^{-}$.
 \end{itemize}

Theorem 1 is known to imply the following:
\begin{itemize}
\item The Brumer--Stark conjecture and the higher rank Brumer--Stark conjecture due to Rubin \cite{rubin}.
\item The integral Gross-Stark conjecture stated by Gross as a refinement to his $p$-adic Stark conjecture \cite{gross}, and 
the higher rank version stated by Popescu and called the Gross--Rubin--Stark conjecture \cite{pcmi}*{Conjecture 5.3.3}.
\item The minus part of the ``refined class number formula for $\G_m$'' due independently to Mazur and Rubin \cite{mr} and Sano \cite{sano}.
\end{itemize}

Write $n=[F:\Q]$.  Using the refined construction of Ritter-Weiss modules at $p=2$ given in our previous work \cite{bsapet}, one obtains stronger versions of all of these results, where $\Theta_{S,T}$ is replaced by
\[ \begin{cases}
\Theta_{S,T}/2^{n-2} & \text{ if } K/F \text{ is unramified at all finite places,} \\
\Theta_{S,T}/2^{n-1} & \text{ if } K/F \text{ is ramified at some finite place.}
\end{cases} \]
This is discussed in \S\ref{s:con1}.

\subsection{Prior work and potential generalization}

There is a significant literature on this case of the equivariant
Tamagawa number conjecture dating back at least two decades.
For a finite abelian extension $K/\Q$, Burns and Greither
\cite{bg} proved the ETNC for nonpositive Tate twists after
inverting $2$, and Flach \cite{Fl11} completed the proof at the
prime $2$. Together with the compatibility with functional
equations established by Burns and Flach \cite{bf06}, these
results prove the full integral ETNC for
\[
 \bigl(h^0(\operatorname{Spec}K)(r),\Z[\Gal(K/F)]\bigr)
\]
for every integer $r$ and every subfield $F\subseteq K$.
In particular, their results at $s=0$ include both the plus
and minus parts.

For abelian CM extensions of general totally real fields,
the approach before our proof of the strong Brumer--Stark
conjecture in \cite{dk} and \cite{bsapet} was to deduce results
from the equivariant Iwasawa main conjecture (EIMC), typically
under a vanishing hypothesis for the relevant Iwasawa
$\mu$-invariant. Significant partial results toward the
$\textrm{ETNC}^{-}$ were obtained in these prior works,
especially in the work of Nickel \cite{Nic11a}, \cite{Nic11b},
\cite{Nic16}, and Burns \cite{burns}.
See also the discussion after \cite{Nic24}*{Theorem 2}.

The key obstacle to deducing the $\textrm{ETNC}^{-}$ at $s=0$ from the EIMC in prior works was the presence of {\em trivial zeroes}.  To be precise, if $S(K)$ contains a finite place $v$ that splits completely in a CM subfield $E$ containing $F$, then the Euler factor $1 - \sigma_v^{-1}$ vanishes in $\Z[\Gal(E/F)]$, forcing $\Theta_{S(K), T}(E/F) = 0$.
Therefore, descending  $\Theta_{S(K), T}(K/F)$  to $E$ yields no information to compare directly with $\Theta_{S(E), T}(E/F)$.  
In other words, the main difficulty is that the group ring element $1 - \sigma_v^{-1}$ is a zero-divisor and cannot be inverted.
This is the feature that we overcome in this article using the strong Brumer--Stark conjecture: by passing to a suitable inverse limit, the discrepancy between strong Brumer--Stark and the $\textrm{ETNC}^{-}$ becomes a non-zerodivisor and hence can be inverted.

Following our work on strong Brumer--Stark, three important papers applied these results to deduce significant partial progress towards the $\textrm{ETNC}^{-}$.  
In the first of these \cite{jn}, Johnston and Nickel prove the $p$-part of the EIMC (for $p$ odd)
for totally real fields for every admissible one-dimensional $p$-adic Lie extension
whose Galois group has an abelian Sylow $p$-subgroup. Moreover, they then use this
to prove the $p$-plus (resp.~$p$-minus) part of the $\textrm{ETNC}$ at negative odd 
(resp.~negative even) integers for Galois CM extensions whose Galois groups have an abelian Sylow $p$-subgroup.
These results use a weaker version of the strong Brumer-Stark conjecture that assumes that $S$ contains all $p$-adic places.
A crucial point of our present article is that we are able to fix our depletion set $\Sigma$ to only contain the primes above $p$ that ramify in $K$.
Note that \cite{jn} avoids $s=0$ because of the presence of trivial zeroes.

Using the strong Brumer--Stark results of \cite{dk,bsapet},
Kurihara \cite{Kur25}*{Theorem 4.1} proves the equivariant
Iwasawa main conjecture for abelian extensions of totally
real fields for all primes $p$, including $p=2$.

There are two other articles that use our work on strong Brumer--Stark to deduce significant partial progress towards the $\textrm{ETNC}^{-}$.
For an odd prime $p$, Atsuta and Kataoka \cite{ak} deduce
many new cases of the $p$-part of $\textrm{ETNC}^{-}$,
for example when $K/F$ is at most tamely ramified at a
prime above $p$.
Nickel \cite{Nic24}*{Theorem 2} proves the $p$-part of
$\textrm{ETNC}^{-}$ when every $p$-adic prime of $F$ is
at most tamely ramified in $K$ or has complex conjugation
in its decomposition group.

Nickel \cite{Nic24}*{Theorem 1} proves the $p$-part of the
strong Stark conjecture for all totally odd Artin characters
when $p$ is odd, without an abelian hypothesis. Combining
Theorem~\ref{t:ETNC} with \cite{Nic24}*{Proposition 1.6}
gives the remaining case $p=2$.

For odd $p$, Theorem~\ref{t:ETNC}, together with the reduction
in \cite{Nic16}*{Proposition 6.2} as used in \cite{Nic24}*{\S3},
also implies the $p$-minus part of the ETNC for finite Galois
CM extensions whose Galois group has an abelian Sylow
$p$-subgroup. We do not address the extension of this
non-abelian reduction to $p=2$ here.
Note also that in the non-abelian case, Nickel \cite{Nic24} proves the $p$-part of the $\textrm{ETNC}^{-}$ for arbitrary Galois CM-extensions assuming that $p \nmid 2|G|$.

\subsection{Altering the smoothing and depletion sets}

In the remainder of the introduction, we sketch the proof of Theorem~\ref{t:ETNC}.    In \S\ref{s:red1} we prove the following lemma.

\begin{lemma}  \label{l:pbyp}
 In order to prove Theorem~\ref{t:ETNC}, it suffices to prove the statement tensored with $\Z_p$ (Theorem~\ref{t:ETNCp} below) for every prime $p$.
\end{lemma}

We note by a subscript $p$ the tensor product with $\Z_p$.  Since $V_{S'}^T(K)$ is projective with constant local rank $(\#S' - 1)$, the module $V_{S'}^T(K)_p$ is free of rank $(\#S' - 1)$ over $\Z_p[G_K]$.

\begin{theorem}  \label{t:ETNCp} Let $p$ be a prime.
Fix bases for the domain and codomain of $f_{K,p,-} = f_{S}^{T}(K)_{p,-}$, thereby inducing bases for the domain and codomain of $f_{E,p,-} = f_{S(E)}^{T}(E)_{p,-}$ for all CM fields $E\subset K$ containing $F$.  
There exists a unit $y_K \in (\Z_p[G_K]_-)^\times$ such that
\begin{equation} \label{e:ETNCp}
 \det(f_{E,p,-}) = \red_{K/E}({y}_K) \cdot \Theta_{S(E),T}(E/F) \text{ in } \Z_p[G_E]_- 
 \end{equation}
for all $E$.
\end{theorem}

Note that we are working with all primes $p$, including $p=2$. The most subtle aspect of Lemma~\ref{l:pbyp} is that Theorem~\ref{t:ETNCp} implies the freeness of $V_{S'}^T(K)_-$ over $\Z[G_K]_-$.

To prove Theorem~\ref{t:ETNCp}, we apply the results of \cite{bsapet}.  For this, we must alter the smoothing set $T$ and depletion set $S$ to mirror the setting there. Define
\begin{equation} \label{e:sigmadef}
 \Sigma(K) = S_\infty \cup \{v \in S_{\ram}(K/F)\colon v \mid p\}, \quad \Sigma'(K) = T \cup  \{v \in S_{\ram}(K/F)\colon v \nmid p\}. 
 \end{equation}
Write $\Sigma  =\Sigma(K), \Sigma' = \Sigma'(K)$.
As above, there is a $\Z[G_K]$-module $V_{S'}^{\Sigma'}(K)$ and a canonical map
\[ f_{\Sigma}^{\Sigma'}(K) \colon V_{S'}^{\Sigma'}(K) \longrightarrow B_{S'}(K). \]
The Ritter--Weiss module $\nabla_{\Sigma}^{\Sigma'}(K)$ is defined to be the cokernel of the map $f_{\Sigma}^{\Sigma'}(K).$

We prove in \cite{dk} that the $\Z_p[G_K]$-module $V_{S'}^{\Sigma'}(K)_p$ is free of rank 
$(\#S' - 1)$.  Therefore $\nabla_{\Sigma}^{\Sigma'}(K)_{p}$ is quadratically presented\footnote{A module is called {\em quadratically presented} if it is isomorphic to the cokernel of a map of free modules of the same rank.  The Fitting ideal of such a module is principal, generated by the determinant of such a map.}, and we  proved  (in \cite{dk} for $p$ odd and in \cite{bsapet} for $p=2$):
\begin{theorem} \label{t:fitt}
We have \begin{equation} \label{e:fitt}
\Fitt_{\Z_p[G_K]_-}(\nabla_{\Sigma}^{\Sigma'}(K)_{p,-}) = (\Theta_{\Sigma, \Sigma'}(K/F)). \end{equation}
\end{theorem}

  The left side of (\ref{e:fitt}) is by definition the image in $\Z_p[G_K]_-$ of the principal ideal generated by the determinant of $f_{\Sigma}^{\Sigma'}(K)_p$ with respect to any choice of bases for its domain and codomain. As above, fixing such bases induces canonical bases for the domain and codomain of 
\[ f_{\Sigma(E)}^{\Sigma'(K)}(E)_p \colon V_{S'}^{\Sigma'(K)}(E)_p \longrightarrow B_{S'}(E)_p \]
for any CM subfield $E \subset K$ containing $F$.  Note that we have altered the depletion set $\Sigma$, but have intentionally left the smoothing set $\Sigma'$
and the auxiliary set $S'$ the same to ensure that the co-invariance property mentioned above holds.  Our altered version of the $p$-part of the $\textrm{ETNC}^{-}$ is the following.

\begin{theorem}  \label{t:ETNCa}
Fix bases for the domain and codomain of $f_K = f_{\Sigma}^{\Sigma'}(K)_p$, thereby inducing bases for the domain and codomain of $f_E = f_{\Sigma(E)}^{\Sigma'(K)}(E)_p$ for all CM fields $E\subset K$ containing $F$.  
There exists a unit $y_K \in (\Z_p[G_K]_-)^\times$ such that
\begin{equation}
\label{e:ETNCa}
 \det(f_{E,-}) = \red_{K/E}({y}_K) \cdot \Theta_{\Sigma(E),\Sigma'(K)}(E/F)  \text{ in } \Z_p[G_E]_- \end{equation}
for all $E$.
\end{theorem}

In \S\ref{s:red2}, we show:

\begin{lemma}  \label{l:alter} Theorem~\ref{t:ETNCa} implies Theorem~\ref{t:ETNCp}.
\end{lemma}

\subsection{Euler systems}

In view of the discussion above, to prove Theorem~\ref{t:ETNC} it suffices to prove that Theorem~\ref{t:fitt} implies Theorem~\ref{t:ETNCa}.
Theorem~\ref{t:fitt} (or more precisely, an easily deduced corollary that replaces the smoothing set $\Sigma'(E)$ by $\Sigma'(K)$) implies that for each $E \subset K$ we have an element $y_E \in (\Z_p[G_E]_-)$ such that
\begin{equation} \label{e:fy}
 \det(f_{E,-}) = {y}_E \cdot \Theta_{\Sigma(E),\Sigma'(K)}(E/F). \end{equation}
 Moreover, we have that $\chi(y_E)$ is a unit in $\Z_p[\chi]^\times$ for any odd character of $G_E$ such that
 \[ \chi(\Theta_{\Sigma(E), \Sigma'(K)}(E/F))  = L_{\Sigma(E), \Sigma'(K)}(\chi^{-1}, 0) \neq 0. \]
 We need to show that the $y_E$ can be chosen to equal $\red_{K/E}({y}_K)$, and that $y_K$ is a unit in $\Z_p[G_K]_-$.
 
 \begin{remark} An element  $y_K \in \Z_p[G_K]_{-}$ is a unit if and only if $\chi(y_K) \in \Z_p[\chi]$ is a unit for every odd character $\chi$ of $G_K$.
Theorem~\ref{t:fitt} implies that this holds for each $\chi$ such that \begin{equation}
\chi(\Theta_{\Sigma(K),\Sigma'(K)}(K/F)) = L_{\Sigma(K), \Sigma'(K)}(\chi^{-1}, 0) \neq 0. \label{e:tz} \end{equation}
 By (\ref{e:tate}), this is the case unless $\chi(G_v) = 1$ for some $v \in \Sigma(K)$.  As we now explain, even if (\ref{e:tz}) fails, we can still deduce that $\chi(y_K) \in \Z_p[\chi]^\times$ for such $\chi$  once we establish the compatibility \[ \red_{K/E}(y_K) = y_E  \text{ for each } E \subset K. \]
 Indeed, take $E$ to be the fixed field of the kernel of $\chi$.
 Since each $w \in \Sigma(E)$ is either archimedean or ramified in $E$, we have $\chi(G_w) \neq 1$ for all such $w$, whence 
 \[ \chi(\Theta_{\Sigma(E), \Sigma'(K)}(E/F))  = L_{\Sigma(E), \Sigma'(K)}(\chi^{-1}, 0) \neq 0 \]
  by (\ref{e:tate}). Therefore, $\chi(y_E) \in \Z_p[\chi]$ is a unit.  The relation $\red(y_K) = y_E$ implies that $\chi(y_K) = \chi(y_E)$, giving the desired result.
 As this holds for all odd characters $\chi$ of $G_K$, we conclude that $y_K$ is a unit in $\Z_p[G_K]_{-}$. 
 \end{remark}
 
As we explain below, functorial properties imply that the $y_E$ satisfy certain norm compatibilities, also known as {\em Euler system relations}.
To describe this, it is convenient to work for each $E$ with the quotient of $\Z_p[G_E]_-$ on which $\Theta_{\Sigma(E),\Sigma'(K)}(E/F)$ is a non-zerodivisor.  To this end, define \begin{equation} \label{e:redef} R_E = e_E \cdot \Z_p[G_E]_- \subset \Q_p[G_E]_-, \end{equation} where 
\begin{equation} \label{e:eedef}
 e_E = \prod_{v \in \Sigma(E)}(1 - \N G_{E,v}/\# G_{E,v}) \end{equation}
 is the idempotent corresponding to the factor on which $\Theta_{\Sigma(E),\Sigma'(K)}(E/F)$  is supported.  Here $G_{E,v} \subset G_E$ denotes the 
 decomposition group at $v$, and \[ \N G_{E,v} = \sum_{\sigma \in G_{E,v}} \sigma \] is its norm. 
 It suffices to prove (\ref{e:ETNCa}) after projecting to $R_E$, since it is easy to show that both sides of the equation vanish under projection to $(1 - e_E) \Z_p[G_E]_- $.
 In  \S\ref{s:norm}, we  prove:
 
\begin{prop} \label{p:euler} 
For every pair of CM fields $E', E$ such that $F \subset E \subset E' \subset K$, we have
 \begin{equation} \label{e:twoys}
 e_E(\red_{E'/E}({y}_{E'}) - y_E)  P(E'/E)  = 0  \text{ in } R_E, \end{equation}
 where
 \begin{equation} \label{e:pdef}
    P(E'/E) = \prod_{v \in \Sigma(E') - \Sigma(E)} (1 - \sigma_v^{-1}),
    \end{equation}
    and $\sigma_v$ denotes the arithmetic Frobenius. 
     \end{prop}
 The troublesome Euler factor $P(E'/E)$, which  can be a zerodivisor,
is all that prevents one from immediately deducing the $\textrm{ETNC}^{-}$. This leads to the following definition.

  \begin{definition}  Consider a set of elements $(z_E)_{E \subset K}$ with $z_E \in R_E$, indexed by the CM fields $E \subset K$ containing $F$.  We say this family  is {\em norm compatible} if for every pair of CM fields $E', E$ such that $F \subset E \subset E' \subset K$, we have
  \[ (e_E \red_{E'/E}(\widetilde{z}_{E'}) - z_E )P(E'/E) = 0  \text{ in } R_E. \]
\end{definition}
 
   Here 
 $\widetilde{z}_{E'}$ denotes any lift of $z_{E'}$ to $\Z_p[G_{E'}]_-$. While $\red_{E'/E}(\widetilde{z}_{E'})$ depends on the lift chosen (even its image in $R_E$ depends on this choice), it is easy to check that $e_E \red_{E'/E}(\widetilde{z}_{E'})P(E'/E)$ is independent of this choice, so the definition is well-formed.
 
With this definition, Proposition~\ref{p:euler} states that $(e_E y_E)_{E \subset K}$ is a norm compatible family.

\begin{definition} Suppose we are given a norm compatible set of elements $(z_E)_{E \subset K}$, with $z_E \in R_E$. Let $K'/F$ be a CM abelian extension such that $K' \supset K$.  We say that the family $(z_E)_{E \subset K}$ {\em can be extended to} $K'$ if there exists a  norm compatible family $(z_E')_{E \subset K'}$ with $z'_E \in R_E$ such that $z'_E = z_E$ for $E \subset K$.
\end{definition}

In the language of \cite{bbds}, the result below describes the ``scarcity of Euler systems." 

\begin{prop} \label{p:main}
Let $D$ be a finite set of finite places of $F$, disjoint from $S(K)$.
Suppose that $(z_E)_{E\subset K}$ is a norm compatible family, with
$z_E\in R_E$. Assume that this family extends to every finite abelian
CM extension $K'/F$ satisfying
\[
 K'\supset K,\qquad \Sigma(K')=\Sigma(K),\qquad
 K'/F\text{ is unramified at every place in }D.
\]
Then there exists a lift $\widetilde z_K\in\Z_p[G_K]_-$ of $z_K$ such that
\[
 e_E\red_{K/E}(\widetilde z_K)=z_E\quad\text{in }R_E
\]
for every CM field $E$ with $F\subset E\subset K$.
\end{prop}

Proposition~\ref{p:main} is proved in \S\ref{s:trivial}.
For its application, take
\[
 D=T\cup\bigl(S'\setminus\Sigma(K)\bigr).
\]
The conditions on $S'$ imply that $D$ is a finite set of finite places
disjoint from $S(K)$. In Proposition~\ref{p:extends} of \S\ref{s:norm},
we show that $(e_Ey_E)_{E\subset K}$ extends to every finite abelian
CM extension $K'/F$ containing $K$ for which
$\Sigma(K')=\Sigma(K)$ and $K'/F$ is unramified at $D$.
We also prove Proposition~\ref{p:euler} there. These statements yield
the implication from Theorem~\ref{t:fitt} to Theorem~\ref{t:ETNCa}.

\bigskip

We are of course indebted to Dominik Bullach, David Burns, Alexandre Daoud, and Soogil Seo for their paper \cite{bbds} that inspired this one.  We are also grateful to Masato Kurihara for providing a careful reading of an earlier draft and making suggestions that greatly improved the exposition.  We would also like to thank Matthew Honnor, Cristian Popescu, and Takamichi Sano for suggestions on an earlier draft. We also thank the anonymous referee for many helpful comments, in particular that the freeness of $V_{S'}^T(K)_-$ (proved in  \S\ref{s:red1} below) is a conjecture of Kurihara.  
The first named author is supported by a grant from the National Science Foundation (DMS-2200787). The second named author is supported by DST-SERB grant SB/SJF/2020-21/11, SERB MATRICS grant MTR/2020/000215, SERB SUPRA grant SPR/2019/ 000422, DST FIST program - 2021 [TPN - 700661], and Tata Education and Development Trust. 

\bigskip
\noindent{\em AI acknowledgement.}
Following submission and detailed review by multiple referees,
we used ChatGPT 6-Pro to help identify remaining errors.
Most of the resulting corrections concerned typographical errors
or minor mathematical issues admitting simple local repairs.
The AI review also identified an important gap concerning the
admissibility of the auxiliary extensions $K'$ considered in
\S\ref{s:ncf}. We addressed this by requiring these extensions
to be unramified at the primes in the fixed finite set
$D=T\cup(S'\setminus\Sigma(K))$ and incorporating this restriction
throughout the argument. 

\section{Definitions and functorial properties}

In this section we fix notation by recalling the definitions of the modules $V_{S'}^{\Sigma'}$, $B_{S'}$, and the map $f_{\Sigma}^{\Sigma'}$.  We then establish some (well-known) elementary properties.

\subsection{Ritter--Weiss modules} \label{s:rwm}

In this section we  let $K/F$ denote an arbitrary Galois extension of number fields with abelian Galois group $G= \Gal(K/F)$.  Let $\Sigma$ and $\Sigma'$ denote finite disjoint sets of places of $F$.  
We impose the condition that $K$ contains no nontrivial roots of unity congruent to 1 modulo all the primes above those in $\Sigma'$.
We assume in this paper that \begin{equation}
 \Sigma \supset S_\infty  \quad \text{and} \quad   \Sigma \cup \Sigma' \supset S_{\ram}(K/F),  \label{e:sigmacond}
 \end{equation} though there is a more general construction that holds in the absence of these conditions (see \cite{bsapet}).
We will recall the definition of the Ritter--Weiss module $\nabla_\Sigma^{\Sigma'}(K)$, referring the reader to \cite{dk}*{Appendix A} or of course the original paper of Ritter--Weiss \cite{rw} for details.

Let $S'$ be a finite set of places of $F$ satisfying the following properties:
\begin{itemize}
\item $S' \supset \Sigma$ and $S' \cap \Sigma' = \emptyset$.
\item $\Cl_{S'}^{\Sigma'}(K) = 1$.  Here $\Cl_{S'}^{\Sigma'}(K)$ is the usual $(S', \Sigma')$-class group of $K$ defined in \cite{dk}*{pg.~292}.
\item $\cup_{w \in S'_K} G_w = G$, where $G_w \subset G$ is the decomposition group at $w$.
\end{itemize}
Here $S'_K$ denotes the set of places of $K$ above those in $S'$.

For each place $v$ of $F$, we fix a place $w$ of $K$ above $v$.  Ritter and Weiss define a $\Z[G_w]$-module $V_w$ sitting in an exact sequence:
\begin{equation} \label{e:vseq}
\begin{tikzcd}
 0 \ar[r] &   K_w^\times \ar[r] &  V_w \ar[r] &  \Delta G_w \ar[r] &  0,
 \end{tikzcd}
\end{equation}
where  $\Delta G_w \subset \Z[G_w]$ denotes the augmentation ideal  (see \cite{rw} or \cite{dk}*{Equation (129)}).
For $w$ finite, they define a $\Z[G_w]$-module $W_w$ sitting in an exact sequence (see \cite{rw}*{\S3}):
\begin{equation} \label{e:wseq}  \begin{tikzcd}
  0  \arrow[r] &  \cO_w^\times \arrow[r] &  V_w \arrow[r] &  W_w \ar[r] &  0. 
  \end{tikzcd}
 \end{equation}
We have
 \begin{equation} \label{e:rww}
  W_w \cong \{(r,s) \in \Delta G_w \times \Z[G_w/I_w] \colon \overline{r} = (1 - {\sigma}_w^{-1})s \}, \end{equation}
where $I_w$ denotes the inertia subgroup of $G_w$ and $\sigma_w \in G_w/I_w$ denotes the arithmetic Frobenius.  Note that if $w$ is unramified then projection on to the second coordinate yields an isomorphism $W_w \cong \Z[G_w]$.  

\begin{remark}
Here and throughout this paper, we follow Kurihara's conventions \cite{kurihara} using the choice of $1 - \sigma^{-1}$, rather than  $\sigma - 1$ as in our previous paper \cite{dk}, which followed Ritter--Weiss \cite{rw}.  These differ only by unit multiples, which had no effect in \cite{dk} because we were working with ideals; here, we are proving exact equalities, and Kurihara's conventions yield cleaner formulas.
\end{remark}

 We adopt the following notation of \cite{gr} and \cite{dk}: for a collection of $G_w$-modules $M_w$, we define 
 \[ \prod_{v}^\sim M_w = \prod_v \Ind_{G_w}^{G} M_w. \]
 
Let $U_w^1 \subset \cO_w^\times$ denote the group of 1-units.  
 Define
\[
 \widetilde{V}_{S'}^{\Sigma'}(K) = \prod_{v \not \in S'  \cup \Sigma'}^{\sim} \cO_w^\times \times \prod_{v \in S'}^{\sim} V_w  \times \prod_{v \in \Sigma'}^{\sim} U_w^1.
\]
Let $\A_K$ denote the ring of adeles of $K$. There is a canonical extension (see pg.~148 of \cite{rw}) 
\begin{equation} \label{e:cog}
\begin{tikzcd}
 0 \ar[r] &  C_K = \A_K^\times/K^\times \ar[r] &  \fO \ar[r] &  \Delta G \ar[r] &  0  
 \end{tikzcd}
\end{equation}
associated to the global fundamental class in $H^2(G, C_K)$, and a canonical surjective map
\[ \theta_V \colon \widetilde{V}_{S'}^{\Sigma'}(K) \lra \fO. \] 
We define
\[ V_{S'}^{\Sigma'}(K) = \ker(\theta_V). \]
 
Next we let \[   \widetilde{B}_{S'}(K) = \prod_{v \in S'} \Z[G_K]. \]
We define a map 
\[ \theta_B \colon \widetilde{B}_{S'}(K) \lra \Z[G_K] \]
componentwise as follows.
\begin{itemize}
\item If $v \in \Sigma$, the component of $\theta_B$ at $v$ is the identity.
\item If $v \in S' \setminus \Sigma$, the component of $\theta_B$ at $v$ is multiplication by $1 - \sigma_v^{-1}$, where $\sigma_v$ is the arithmetic Frobenius (note that by our assumptions, $v \in S' \setminus \Sigma$ implies that $K/F$ is unramified at $v$).
\end{itemize}
We then define
\[ B_{S'}(K) = \ker(\theta_B). \]
To proceed, we must choose a splitting of the map $\theta_B$.  We do this once and for all by picking an infinite place $\infty$ of $F$, since the component of $\theta_B$ at $\infty$ is the identity.  This yields an isomorphism
\begin{equation} \label{e:bisom}
 B_{S'}(K) \cong \prod_{v \in S' \setminus \{\infty\}} \Z[G_K]. \end{equation}
Next we define 
\[ \widetilde{f}_{\Sigma}^{\Sigma'}(K) \colon  \widetilde{V}_{S'}^{\Sigma'}(K) \lra  \widetilde{B}_{S'}(K) \]
 componentwise.
\begin{itemize}
\item If $v \in \Sigma$, the component of  $\widetilde{f}_{\Sigma}^{\Sigma'}(K)$ at $v$ is induced by the composition of the map $V_w \rightarrow \Delta G_w$ with the canonical injection $\Delta G_w \subset \Z[G_w]$.
\item If $v \in S' \setminus \Sigma$, the component of  $\widetilde{f}_{\Sigma}^{\Sigma'}(K)$ at $v$ is induced by the composition of $V_w \rightarrow W_w$ with the isomorphism $W_w \cong \Z[G_w]$ given by projection on to the second coordinate, which as noted above is an isomorphism since $v$ is unramified.
\item If $v \not \in S'$, the component of  $\widetilde{f}_{\Sigma}^{\Sigma'}(K)$ at $v$ is $0$.
\end{itemize}
It can be shown that $\widetilde{f}_{\Sigma}^{\Sigma'}(K)$ restricts to a map
\[ f_{\Sigma}^{\Sigma'}(K) \colon  {V}_{S'}^{\Sigma'}(K) \lra  {B}_{S'}(K). \]
The module $\nabla_{\Sigma}^{\Sigma'}(K)$ is defined to be the cokernel of $f_{\Sigma}^{\Sigma'}(K)$.

\subsection{Functorial properties}\label{s:functoriality}

  In this section, we recall some properties of the map $f_{\Sigma}^{\Sigma'}$ described above.
Let $E/F$ denote any field extension  contained in $K$, and write $G_E = \Gal(E/F)$. 
Fix a prime $p$.
Suppose that $\Sigma$ and $\Sigma'$ satisfy (\ref{e:sigmacond}), and furthermore that $\Sigma'$ contains no primes above $p$ that are ramified in $K$.
We prove in \cite{dk} that under these conditions, the modules $V_{S'}^{\Sigma'}(E)_p$ and $B_{S'}(E)_p$ are free $\Z_p[G_E]$-modules of rank $(\#S' - 1)$. 
After choosing bases, it therefore makes sense to consider the determinant of the map
\[ f_{\Sigma}^{\Sigma'}(E)_p \colon  V_{S'}^{\Sigma'}(E)_p \lra B_{S'}(E)_p. \]

We now consider three specific cases fitting within this framework:
\begin{enumerate}
\item $f_{\Sigma}^{\Sigma'}(E)_p$, where $\Sigma = \Sigma(K), \Sigma' = \Sigma'(K)$ are defined in (\ref{e:sigmadef}).
\item $f_{\Sigma-\{v\}}^{\Sigma'}(E)_p$, where $v \in \Sigma$ is unramified in $E/F$.
\item $f_{\Sigma}^{\Sigma' \cup \{ v \}}(E)_p$, where $v \notin S' \cup \Sigma'$.
\end{enumerate}
In items (2) and (3), $(\Sigma,\Sigma')$ may be any pair
satisfying the standing hypotheses, not necessarily the pair
in item (1). In item (2), $v$ is finite, and the modified data
are considered over $E/F$, where $v$ is unramified.
We use the same auxiliary set $S'$ in all three cases.

Using the isomorphism (\ref{e:bisom}), it is clear that a choice of basis for $B_{S'}(K)$ induces a basis for $B_{S'}(E)$ by taking $\Gal(K/E)$-coinvariants; in fact the isomorphism yields a canonical basis.

Similarly, a choice of basis for each of $V_{S'}^{\Sigma'}(K)_p$ and 
$V_{S'}^{\Sigma' \cup \{ v \}}(K)_p$ induces bases for $V_{S'}^{\Sigma'}(E)_p$ and $V_{S'}^{\Sigma' \cup \{ v \}}(E)_p$  by taking co-invariants. This follows from the following.
\begin{lemma}\label{l:coinv}
The map $f_\Sigma^{\Sigma'}(E)_p$ is obtained from
$f_\Sigma^{\Sigma'}(K)_p$ by taking $\Gal(K/E)$-co-invariants.
\end{lemma}
\begin{proof}
Apply the argument of \cite{dk}*{Appendix B.1} after tensoring with
$\Z_p$. The standing hypotheses imply condition (A2') of
\cite{dk}*{Appendix A.3}. More explicitly, at a ramified place in
$\Sigma'$ of residue characteristic $\ell$, we have $\ell\neq p$.
As in the proof of \cite{dk}*{Lemma A.4}, the group of $1$-units
contains a cohomologically trivial open subgroup of $\ell$-power
index; its inclusion becomes an isomorphism after tensoring with
$\Z_p$. Thus the norm and co-invariant argument of Appendix B.1
applies under the present hypotheses and is compatible with the
local maps defining $f_\Sigma^{\Sigma'}$.
\end{proof}
 
\begin{prop}\label{p:functoriality} We have:
\begin{enumerate}
\item $\det(f_{\Sigma}^{\Sigma'}(E)_p) = \red_{K/E}(\det(f_{\Sigma}^{\Sigma'}(K)_p))$
\item $\det(f_{\Sigma}^{\Sigma'}(E)_p) = (1 - \sigma_v^{-1}) \det(f_{\Sigma-\{v\}}^{\Sigma'}(E)_p)$
\item Suppose $v \not\in S' \cup \Sigma'$.  Then \[ \det(f_{\Sigma}^{\Sigma' \cup \{ v \}}(E)_p) = \red_{K/E}({x}_K)(1 - \sigma^{-1}_v \N v)  \det(f_{\Sigma}^{\Sigma'}(E)_p) \] for some $x_K \in \Z_p[G_K]^\times$ that does not depend on $E$.
\end{enumerate}
\end{prop}
\begin{proof} 

(1) This follows immediately from Lemma \ref{l:coinv}.  

\medskip

(2) Consider the map of complexes:

\[ 
\begin{tikzcd}[row sep=large, column sep = huge]
V_{S'}^{\Sigma'}(E)_p \arrow{r}{f_{\Sigma-\{v\}}^{\Sigma'}(E)_p} \arrow{d} & B_{S'}(E)_p \arrow{d} \\
 V_{S'}^{\Sigma'}(E)_p \arrow{r}{f_{\Sigma}^{\Sigma'}(E)_p}& B_{S'}(E)_p. \end{tikzcd} \]

The first vertical arrow is the identity, while the second vertical arrow is equal to \[ (\id, 1-\sigma_v^{-1}) \colon N \oplus \Z_p[G_E] \longrightarrow N \oplus \Z_p[G_E], \] where $N$ is a free $\Z_p[G_E]$-module. This has determinant $1-\sigma_v^{-1}$, and so we conclude that \[ \det(f_{\Sigma}^{\Sigma'}(E)_p) = (1 - \sigma_v^{-1}) \det(f_{\Sigma-\{v\}}^{\Sigma'}(E)_p). \]

\medskip

(3) Consider the map of complexes 
\[ \begin{tikzcd}[row sep=large, column sep = huge]
V_{S'}^{\Sigma' \cup \{ v \}}(E)_p \arrow{r}{f_{\Sigma}^{\Sigma' \cup \{ v \}}(E)_p} \arrow{d}& B_{S'}(E)_p \arrow{d} \\
V_{S'}^{\Sigma'}(E)_p \arrow{r}{f_{\Sigma}^{\Sigma'}(E)_p} & B_{S'}(E)_p. \end{tikzcd}
\] 

The second vertical arrow is the identity.  This map of complexes is injective, with cokernel
\[  (\Ind_{\Z[G_{E,w}]}^{\Z[G_E]}\F_w^\times)_p \longrightarrow 0, \]
where $w$ is a place of $E$ over $v$ and $\F_w$ is its residue field.  Here $G_{E,w} \subset G_E$ denotes the decomposition group at the place $w$ above $v$.
This uses the fact that $v \not \in S' \cup \Sigma'$ and $\cO_w^\times/U_w^1 \cong \F_w^\times$.

Since the domain and codomain of \[ h_E \colon V_{S'}^{\Sigma' \cup \{ v \}}(E)_p \longrightarrow V_{S'}^{\Sigma'}(E)_p \] are free of the same rank $\#S' - 1$, $h_E$ is a quadratic presentation of $(\Ind_{\Z[G_{E,w}]}^{\Z[G_E]}\F_w^\times)_p$. The Fitting ideal of $(\Ind_{\Z[G_{E,w}]}^{\Z[G_E]}\F_w^\times)_p$ as a $\Z_p[G_E]$-module is $(1 - \sigma^{-1}_v \N v)$. Hence, for any choice of basis for the domain and codomain of $h_E$, we have $\det(h_E) = x_E  (1 - \sigma^{-1}_v \N v)$ for some $x_E \in \Z_p[G_E]^\times$. 

Note that $h_E$ equals the $\Gal(K/E)$ co-invariants of the map $h_K$. Thus if we use bases for the domain and codomain of $h_E$ induced from a choice of bases for $h_K$, then we may take $x_E = \red_{K/E}({x}_K)$.
We conclude that \[ \det(f_{\Sigma}^{\Sigma' \cup \{ v \}}(E)_p) = \red_{K/E}({x}_K)(1 - \sigma^{-1}_v \N v)  \det(f_{\Sigma}^{\Sigma'}(E)_p). \]
\end{proof}

Consider the maps $f_K = f_{\Sigma(K)}^{\Sigma'(K)}(K)_p$ and $f_E = f_{\Sigma(E)}^{\Sigma'(K)}(E)_p$ from the introduction. From parts 1 and 2 of Proposition~\ref{p:functoriality}, we obtain:
\begin{corollary}
\begin{equation}\label{e:fnorm}
 \red_{K/E}(\det(f_K)) = P(K/E) \det(f_E) .
\end{equation}
\end{corollary}

Put $D=T\cup\bigl(S'\setminus\Sigma(K)\bigr)$, and consider a finite
abelian CM extension $K'/F$ containing $K$ such that
$\Sigma(K')=\Sigma(K)$ and $K'/F$ is unramified at $D$.
Then $S'\cap\Sigma'(K')=\emptyset$. Moreover, $T$ still satisfies the
root-of-unity condition for $K'$: if $T$ contains two primes of
different residue characteristic this is immediate, while in the
single-prime case, of residue characteristic $\ell>[F:\Q]+1$,
unramifiedness at that prime excludes nontrivial $\ell$-power roots
of unity congruent to $1$ there.
For each CM field $E'\subset K'$ containing $F$, consider
\begin{equation}\label{e:gepdef}
g_{E'}=f_{\Sigma(E')}^{\Sigma'(K')}(E')_p\colon
 V_{S''}^{\Sigma'(K')}(E')_p\longrightarrow B_{S''}(E')_p.
\end{equation}
Choose $S''\supset S'$ disjoint from $\Sigma'(K')$ and large enough
to satisfy the auxiliary-set conditions of \S\ref{s:rwm} for $K'$
and its intermediate fields. This is possible since
$S'\cap\Sigma'(K')=\emptyset$. The set $S''$ may be larger than $S'$;
we compare the resulting determinants by the following lemma.

\begin{lemma}\label{l:aux-change}
Let $S'_1 \subset S'_2$ be two auxiliary sets for $(K, \Sigma, \Sigma')$, with maps denoted 
\[ f_\Sigma^{\Sigma'}(K)_i \colon V_{S_i'}^{\Sigma'}(K)_p \lra B_{S'_i}(K)_p, \quad i=1,2. \]
A choice of bases for the domain and codomain of $f_{\Sigma}^{\Sigma'}(K)_1$ naturally determines bases for the domain and codomain of $f_\Sigma^{\Sigma'}(K)_2$, and with these bases, $\det(f_{\Sigma}^{\Sigma'}(K)_1) = \det(f_{\Sigma}^{\Sigma'}(K)_2)$.
\end{lemma}
\begin{proof}
By induction, we add one prime at a time, and consider $S'_1 = S'$, $S'_2 = S' \cup \{ v \}$. The injective map of complexes
\[\begin{tikzcd}[row sep=large, column sep = huge]
V_{S'}^{\Sigma'}(K)_p \arrow{r}{f_{\Sigma}^{\Sigma'}(K)_1} \arrow{d}& B_{S'}(K)_p \arrow{d} \\
V_{S' \cup \{ v \}}^{\Sigma'}(K)_p \arrow{r}{f_{\Sigma}^{\Sigma'}(K)_2} & B_{S'  \cup \{ v \}}(K)_p \end{tikzcd}\]
has cokernel
\[ \begin{tikzcd}[row sep=large, column sep = large] \Z_p[G_K] \arrow{r}{\cong} & \Z_p[G_K]. 
\end{tikzcd} \]

Given a basis of $V_{S'}^{\Sigma'}(K)_p$ and any choice of splitting of the surjection \[ V_{S' \cup \{ v \}}^{\Sigma'}(K)_p \longrightarrow \Z_p[G_K], \] we obtain a basis of $V_{S' \cup \{ v \}}^{\Sigma'}(K)_p$. The comparison of determinants is then clear.
\end{proof}

Define \begin{equation} Q(K'/K) = \prod_{v \in \Sigma'(K') - \Sigma'(K)} (1 - \sigma^{-1}_v \N v). \end{equation}

\begin{corollary}\label{c:q-factor}
There exists a unit $x_K\in\Z_p[G_K]^\times$ such that, for every CM
field $E\subset K$ containing $F$,
\[
 \det(g_E)=\red_{K/E}\bigl(x_KQ(K'/K)\bigr)\det(f_E).
\]
\end{corollary}
\begin{proof}
First enlarge the auxiliary set for $f_E$ from $S'$ to $S''$ using
Lemma~\ref{l:aux-change}. Choose the splittings over $K$ and use their
images under co-invariants over $E$; the resulting determinants are
unchanged. Both maps can now be computed with auxiliary set $S''$.
Every prime in $\Sigma'(K')\setminus\Sigma'(K)$ lies outside
$S''\cup\Sigma'(K)$. Applying Proposition~\ref{p:functoriality}(3)
successively to these primes proves the assertion, with the same
unit $x_K$ before reduction to each $E$.
\end{proof}

\section{Norm compatible families} \label{s:ncf}

We return to the setting of the introduction, which we now recall.  The field $F$ is totally real and $K$ is a finite abelian CM extension of $F$. We define the sets $\Sigma(K)$ and $\Sigma'(K)$ by (\ref{e:sigmadef}) and fix a prime $p$.  For each CM subfield $E \subset K$ containing $F$, we consider
\[ f_{E,-} = f_{\Sigma(E)}^{\Sigma'(K)}(E)_{p,-} \colon V_{S'}^{\Sigma'(K)}(E)_{p,-} \longrightarrow B_{S'}(E)_{p,-}. \]

A choice of bases for the domain and codomain of $f_{K,-}$, which are free $\Z_p[G_K]_-$-modules, yields bases for the domain and codomain of each $f_{E,-}$.
Theorem~\ref{t:fitt} yields elements $y_{E} \in \Z_p[G_E]_{-}$ such that \[ \det(f_{E,-}) = y_E \Theta_{\Sigma(E),\Sigma'(K)}(E/F), \]
with the property that the elements $e_E y_E \in R_E = e_E \Z_p[G_E]_{-}$ are units (see (\ref{e:redef}) and (\ref{e:eedef})).

In this section we first prove Proposition \ref{p:euler}, which states that the family of units $(e_E y_E)_{E \subset K}$ is norm compatible.
Next we prove that this norm compatible family extends over every
finite abelian CM extension $K'/F$ containing $K$ with
$\Sigma(K')=\Sigma(K)$ and unramified at
$D=T\cup\bigl(S'\setminus\Sigma(K)\bigr)$.
Finally, we prove Proposition~\ref{p:main} on the scarcity of norm
compatible families with this restricted extension property.

\subsection{Construction and extension of a norm compatible family}
\label{s:norm}

\begin{proof}[Proof of Proposition \ref{p:euler}]
From the definitions of the Stickelberger elements, we have
\begin{equation}\label{e:tnorm}
 \red_{E'/E}(\Theta_{\Sigma(E'),\Sigma'(K)}(E'/F)) = \Theta_{\Sigma(E),\Sigma'(K)}(E/F) P(E'/E) 
 \end{equation}
 in $\Z_p[G_E]_-.$
Combining (\ref{e:fy}), (\ref{e:fnorm}) and (\ref{e:tnorm}), we find
\begin{equation} \label{e:twoyst}
 (\red_{E'/E}({y}_{E'}) - y_E)  \Theta_{\Sigma(E),\Sigma'(K)}(E/F) P(E'/E)  = 0. \end{equation}
We project to the quotient $R_E$ of $\Z_p[G_E]_-$, on which $ \Theta_{\Sigma(E),\Sigma'(K)}(E/F)$ is a non-zerodivisor and can be canceled.  This yields \[ (e_E \red_{E'/E}({y}_{E'}) - e_E y_E)  P(E'/E)  = 0 \] in $R_E$ as desired.
\end{proof}

\begin{prop}\label{p:extends}
Put $D=T\cup\bigl(S'\setminus\Sigma(K)\bigr)$.
Let $K'/F$ be a finite abelian CM extension containing $K$ such that
$\Sigma(K')=\Sigma(K)$ and $K'/F$ is unramified at $D$.
Then $(e_Ey_E)_{E\subset K}$ extends to a norm compatible family
on the CM subfields of $K'$ containing $F$.
\end{prop}
\begin{proof}
Choose $S''$ as in the paragraph containing (\ref{e:gepdef}).
Fix bases for the domain and codomain of $g_{K'}$, and use the
induced bases over every CM intermediate field $E'$.
Recall that
\[
 g_{E'}=f_{\Sigma(E')}^{\Sigma'(K')}(E')_p\colon
 V_{S''}^{\Sigma'(K')}(E')_p\longrightarrow B_{S''}(E')_p.
\]
Theorem~\ref{t:fitt}, with the same change of smoothing set used
above, gives $w_{E'}\in\Z_p[G_{E'}]_-$ such that
\begin{equation}\label{e:wdef}
 \det(g_{E',-})=w_{E'}\Theta_{\Sigma(E'),\Sigma'(K')}(E'/F),
\end{equation}
and $e_{E'}w_{E'}\in R_{E'}^\times$.
Proposition~\ref{p:euler}, applied to $K'/F$, shows that
$(e_{E'}w_{E'})_{E'\subset K'}$ is norm compatible.

Now let $E\subset K$ be a CM field containing $F$. We have
\begin{equation}\label{e:w1}
 \det(g_{E,-})=
 w_E\Theta_{\Sigma(E),\Sigma'(K)}(E/F)\overline{Q(K'/K)},
\end{equation}
where the bar denotes the image in $\Z_p[G_E]_-$.
Corollary~\ref{c:q-factor}, followed by projection to the minus side,
also gives
\begin{align}\label{e:w2}
 \det(g_{E,-})
 &=\red_{K/E}(x_K)\det(f_{E,-})\overline{Q(K'/K)}\nonumber\\
 &=\red_{K/E}(x_K)y_E
   \Theta_{\Sigma(E),\Sigma'(K)}(E/F)\overline{Q(K'/K)}.
\end{align}
The product of the Stickelberger element and $\overline{Q(K'/K)}$
is a nonzero-divisor on $R_E$. Comparing (\ref{e:w1}) and
(\ref{e:w2}) there yields
\[
 e_Ew_E=e_Ey_E\red_{K/E}(x_K).
\]
Choose a unit $x_{K'}\in\Z_p[G_{K'}]^\times$ lifting
$x_K\in\Z_p[G_K]^\times$; such a lift exists because these are
commutative semilocal rings. Then
\[
 \bigl(e_{E'}w_{E'}\red_{K'/E'}(x_{K'})^{-1}\bigr)_{E'\subset K'}
\]
is norm compatible and restricts to $(e_Ey_E)_{E\subset K}$.
\end{proof}

\subsection{Scarcity of norm compatible families} \label{s:trivial}

In this section, we prove Proposition~\ref{p:main}. 

\begin{lemma}\label{l:ep}
Let $E/F$ be a finite abelian CM extension, let $v\mid p$ with
$v\notin\Sigma(E)$, and let $m$ be a positive integer.
Let $D$ be a finite set of finite places disjoint from
$S(E)\cup\{v\}$. There exists a finite abelian CM extension
$E'/F$ containing $E$ such that
\begin{itemize}
\item $\Sigma(E')=\Sigma(E)$,
\item $\#G_{E',w}=\#G_{E,w}$ for every $w\in\Sigma(E)$,
\item $E'/F$ is unramified at every place in $D$, and
\item the Frobenius at $v$ in $G_{E'}$ has order divisible by $p^m$.
\end{itemize}
\end{lemma}
\begin{proof}
Apply the Grunwald--Wang theorem in the character formulation of
\cite{conrad}*{Proposition A.3} to the finite set consisting
of the infinite places, the places above $p$, and the places in $D$.
At $v$, prescribe an unramified character of order $p^m$.
At every other place in this set, prescribe the trivial character.
There is a finite-order global character with these restrictions;
its order is $p^m$, except possibly $2p^m$ in the exceptional case
when $p=2$. Let $F'/F$ be the cyclic extension it defines.
The trivial restrictions at the real places make $F'$ totally real.
It splits completely at $D$ and at the finite places in $\Sigma(E)$,
and is unramified at every place above $p$.
Set $E'=EF'$. The first three assertions follow from these local
conditions. Since $E'/F$ is unramified at $v$, its Frobenius maps
to the Frobenius of order $p^m$ in $\Gal(F'/F)$, proving the last
assertion.
\end{proof}

Fix $D$ as in Proposition~\ref{p:main}. Call a finite abelian CM
extension $L/F$ admissible if it contains $K$, satisfies
$\Sigma(L)=\Sigma(K)$, and is unramified at $D$.
We first choose the extensions of the given family coherently.
For each admissible $L$, let $\mathcal X_L$ be the set of norm
compatible families on its CM subfields containing $F$ that
restrict to the given family on $K$. This is a nonempty compact
space: it is nonempty by the extension hypothesis in the statement of Proposition~\ref{p:main}. Also, it is compact since the modules $R_E$ are finitely generated $\Z_p$-modules,
and the norm compatibility and restriction conditions are closed
conditions in their finite product. The spaces $\mathcal X_L$
form an inverse system under restriction. Admissible fields are
closed under composita, so every finite collection of compatibility
conditions is satisfied by a family over an admissible compositum.
Compactness therefore gives an element of the inverse limit.
We use this coherent family throughout the induction. In particular,
its restriction to any admissible field has a specified extension
to every larger admissible field.

The proof will proceed by induction using this coherent family.
Write \begin{equation} \label{e:sigmak}
\Sigma(K) = \{v_1, v_2, \dots, v_m \}. 
\end{equation}  For $i = 0, \dotsc, m$, define
\[ P_i(E'/E) = \prod_{\substack{j > i \\ v_j \in \Sigma(E') - \Sigma(E)}} (1 - \sigma_{v_j}^{-1}) \in \Z_p[G_E]_-. \]
So in particular $P_0(E'/E) = P(E'/E)$ as defined earlier, and $P_m(E'/E) = 1$.  
We will prove the following statement by induction on $i$:
\begin{quote} There exists $\widetilde{z}_K \in \Z_p[G_K]_-$ lifting $z_K \in R_K$ such that
\begin{equation} \label{e:pi}
 (e_E \red_{K/E}(\widetilde{z}_K) - z_E)P_i(K/E) = 0 
 \end{equation}
in $R_E$ for all $E \subset K$.
\end{quote}

We prove this statement simultaneously for all admissible fields
in place of $K$, using the restrictions of the coherent family
just chosen and the same ordered set (\ref{e:sigmak}).
For $i=0$, the statement is the norm compatibility relation, so any
lift of $z_K$ works. Let $i\geq1$ and assume the statement for $i-1$
for every admissible field.

\begin{lemma} \label{l:pin} In order to deduce (\ref{e:pi}), it suffices to prove that for each positive integer $k$, 
there exists $\widetilde{z}_{K,k} \in \Z_p[G_K]_-$
lifting $z_K \in R_K$ such that
\begin{equation} \label{e:pin}
 (e_E \red_{K/E}(\widetilde{z}_{K,k}) - z_E)P_i(K/E) \equiv 0 \text{ in } R_E/p^k
 \end{equation}
 for all $E \subset K$.
\end{lemma}

\begin{proof}
Let
\[
A=\mathbf{Z}_p[G_K]_{-},
\]
equipped with its $p$-adic topology, and let
\[
\mathcal{L}
=
\left\{
\widetilde{z}\in A :
e_K\widetilde{z}=z_K
\right\}
\]
be the set of lifts of $z_K$ to $A$. Since $A$ is a finite free
$\mathbf{Z}_p$-module, it is compact. Moreover, $\mathcal{L}$ is a
closed subset of $A$, and is therefore compact.

For each positive integer $k$, define
\[
\mathcal{C}_k
=
\left\{
\widetilde{z}\in\mathcal{L} :
\bigl(
e_E\operatorname{red}_{K/E}(\widetilde{z})-z_E
\bigr)
P_i(K/E)
\in p^kR_E
\text{ for every } E\subset K
\right\},
\]
where $E$ ranges over the CM subfields of $K$ containing $F$.
By the hypothesis of the lemma, $\mathcal{C}_k$ is nonempty for every
positive integer $k$.

Each $\mathcal{C}_k$ is closed in $\mathcal{L}$. Indeed, for each
$E\subset K$, the map
\[
A\longrightarrow R_E,\qquad
\widetilde{z}\longmapsto
\bigl(
e_E\operatorname{red}_{K/E}(\widetilde{z})-z_E
\bigr)
P_i(K/E)
\]
is continuous, and $p^kR_E$ is closed in $R_E$. Moreover, $\mathcal{C}_{k+1}\subseteq\mathcal{C}_k$
for every $k$. Thus the $\mathcal{C}_k$ form a nested sequence of
nonempty closed subsets of the compact space $\mathcal{L}$. It follows
that
\[
\bigcap_{k\geq 1}\mathcal{C}_k\neq\varnothing.
\]

Choose $\widetilde{z}_K$ in this intersection.
Then $\widetilde{z}_K$ lifts $z_K$, and for every $E\subset K$ we have
\[
\bigl(
e_E\operatorname{red}_{K/E}(\widetilde{z}_K)-z_E
\bigr)
P_i(K/E)
\in
\bigcap_{k\geq 1}p^kR_E = 0.
\]
The last equality follows since $R_E$ is a finitely generated $\mathbf{Z}_p$-module, and hence 
$p$-adically separated.
Consequently,
\[
\bigl(
e_E\operatorname{red}_{K/E}(\widetilde{z}_K)-z_E
\bigr)
P_i(K/E)=0
\]
in $R_E$ for every $E\subset K$. This is precisely~\eqref{e:pi}.
\end{proof}

Fix a positive integer $k$. 
 Let $E$ be a CM field with $F \subset E \subset K$ such that $v_i$ is unramified in $E$ (i.e. $v_i \not\in \Sigma(E)$).
 Let $p^r$ denote the power of $p$ dividing $\prod_{v \in \Sigma(E)} \#G_{E,v}$, 
and let $p^s$ be the power of $p$ dividing the order of $\sigma_{v_i}$ in $G_E$. Let $k' = k + r + s$.
Choose $E'$ by Lemma~\ref{l:ep}, with $v=v_i$, $m=k'$, and the
fixed set $D$. Its hypotheses hold since $D\cap S(K)=\emptyset$
and $v_i\in\Sigma(K)$. Let $K'$ be the compositum of $K$ with these
finitely many fields $E'$. Then $\Sigma(K')=\Sigma(K)$ and $K'/F$
is unramified at $D$, so $K'$ is admissible. Use the restriction
of the coherent family to $K'$. The induction hypothesis supplies
a lift $\widetilde z_{K'}\in\Z_p[G_{K'}]_-$ such that
\begin{equation} \label{e:induction}
  (e_L \red_{K'/L}(\widetilde{z}_{K'}) - z_L)P_{i-1}(K'/L) = 0  
  \end{equation}
for all $L \subset K'$. 

We define \[ \widetilde{z}_K = \red_{K'/K}(\widetilde{z}_{K'}) \in \Z_p[G_K]_-.\]  This is a lift of $z_K \in R_K$ by the norm compatibility relation since $\Sigma(K') = \Sigma(K)$. 

If $L \subset K$ is a CM field containing $F$ such that $v_i \in \Sigma(L)$, then $P_i(K/L) = P_{i-1}(K/L)$, and the statement (\ref{e:induction}) can be written
\[ 
  (e_L \red_{K/L}(\widetilde{z}_{K}) - z_L)P_{i}(K/L) = 0.
\]
  This is exactly the statement (\ref{e:pi}) that we are trying to prove for the field $L$.
  It remains to consider the more salient case of the fields $E$ such that $v_i \not \in \Sigma(E)$ which were used in the construction of $K'$.
For $L=E'$ with $E'$ as above, (\ref{e:induction}) reads
\begin{equation} \label{e:inde}
  (e_{E'} \red_{K'/E'}(\widetilde{z}_{K'}) - z_{E'})P_{i}(K'/E')(1 - \sigma_{v_i}^{-1}) = 0.  
  \end{equation}

By the definition of $e_{E'}$ given in (\ref{e:eedef}), and the second property of the extension $E'$ given in Lemma~\ref{l:ep}, we see that $p^r e_{E'} \in \Z_p[G_{E'}]$. Hence, $p^r z \in \Z_p[G_{E'}]_-$ for any $z \in R_{E'} \subset \Q_p[G_{E'}]_-$.  From (\ref{e:inde}) we therefore have
\[  (p^r e_{E'} \red_{K'/E'}( \widetilde{z}_{K'}) - p^r z_{E'})P_{i}(K'/E')(1 - \sigma_{v_i}^{-1}) = 0  \text{ in } \Z_p[G_{E'}]_-.  \]
  The kernel of multiplication by $(1 - \sigma_{v_i}^{-1})$ in $\Z_p[G_{E'}]_-$ is the ideal generated by
 \[ 1 + \sigma_{v_i} + \cdots + \sigma_{v_i}^{n'-1} \]
 where $n'$ is the order of $\sigma_{v_i}$ in $G_{E'}$.
We can therefore write
\[ (p^r e_{E'} \red_{K'/E'}( \widetilde{z}_{K'}) - p^r z_{E'})P_{i}(K'/E') \in (1 + \sigma_{v_i} + \cdots + \sigma_{v_i}^{n'-1}) \text{ in } \Z_p[G_{E'}]_-. \]
Reducing to $\Z_p[G_E]_-$, we find
\begin{align} \label{e:pr}
 \left[p^r e_{E} \red_{K'/E}( \widetilde{z}_{K'}) -  \red_{E'/E}(p^r z_{E'})\right] & P_i(K/E) \\
 & \in (n'/n)(1 + \sigma_{v_i} + \cdots + \sigma_{v_i}^{n-1}) \subset (p^{k'-s}) \text{ in } \Z_p[G_{E}]_-,\nonumber
 \end{align}
 where $n$ is the order of $\sigma_{v_i}$ in $G_E$.
Note that $\red_{E'/E}(p^r {z}_{E'}) = p^r z_E$ by the norm compatibility relation, since $\Sigma(E')= \Sigma(E)$.
Therefore, projecting (\ref{e:pr}) to $R_E$ and canceling $p^r$ we obtain
\begin{equation} \label{e:top}
  (e_{E} \red_{K'/E}( \widetilde{z}_{K'}) - {z}_{E})P_i(K/E) \in (p^{k'-s-r}) = (p^{k}) \text{ in } R_E. \end{equation}

 From (\ref{e:top}) we obtain
\[ 
  (e_{E} \red_{K/E}( \widetilde{z}_{K}) - {z}_{E})P_i(K/E) \equiv 0  \text{ in } R_E/p^k. 
  \]
This is (\ref{e:pin}), so by Lemma~\ref{l:pin} we obtain the desired inductive statement (\ref{e:pi}).  The result for $i=m$ is exactly the statement of Proposition~\ref{p:main}, since $P_m(K/E) = 1$.  This concludes the proof of Proposition~\ref{p:main}.

\section{Reductions} \label{s:reduce}

\subsection{Reduction 1: Lemma~\ref{l:pbyp}}  \label{s:red1}

In this section we prove Lemma~\ref{l:pbyp}.  This has two parts: first, we must show that Theorem~\ref{t:ETNCp} implies the freeness of $V_{S'}^T(K)_-$.  Given this freeness, we must show that the local statements (\ref{e:ETNCp}) imply the global statement (\ref{e:ETNC}).  The second half of this is well-known and straightforward, so we quickly recall the proof.

\begin{proof}[Proof of Equation~(\ref{e:ETNC})]  Assume that $V_{S'}^T(K)_-$ is free, so the rest of the statement of Theorem~\ref{t:ETNC} is well-formed.  Theorem~\ref{t:ETNCp} for all primes $p$ yields an element $y_K \in (\hat{\Z}[G_K]_-)^\times$ such that (\ref{e:ETNC}) holds.
We need to show that $y_K \in (\Z[G_K]_-)^\times$.
For each prime $p$, let $y_{K,p}$ denote the $p$-component of
$y_K$. Let $\chi$ be an odd character of $G_K$, and let
$E\subset K$ be the fixed field of $\ker(\chi)$. Applying
$\chi$ to~\eqref{e:ETNCp}, we obtain
\[
\chi(y_{K,p})
=
\frac{\chi\bigl(\det(f_{S(E)}^T(E)_-)\bigr)}
     {L_{S(E),T}(\chi^{-1},0)}.
\]
The denominator is nonzero by~\eqref{e:tate}, since
$\chi(G_{E,v})\neq 1$ for every $v\in S(E)$.
The right-hand side is the $p$-adic realization of a nonzero
algebraic number that is independent of $p$, and these values
are Galois-equivariant as $\chi$ varies. They therefore define
an element
\[
y_K\in(\Q[G_K]_-)^\times
\]
whose image in $\Q_p[G_K]_-$ is $y_{K,p}$ for every prime $p$.

We now use the elementary identity
\[
\bigl(A\otimes_{\mathbf{Z}}\mathbf{Q}\bigr)^\times
\cap
\bigl(A\otimes_{\mathbf{Z}}\widehat{\mathbf{Z}}\bigr)^\times
=
A^\times
\]
for any finite free $\mathbf{Z}$-algebra $A$, where the
intersection is taken inside
\[
\prod_p \bigl(A\otimes_{\mathbf{Z}}\mathbf{Q}_p\bigr)
\]
and the first factor is embedded diagonally. Indeed,
\[
\bigl(A\otimes_{\mathbf{Z}}\mathbf{Q}\bigr)
\cap
\bigl(A\otimes_{\mathbf{Z}}\widehat{\mathbf{Z}}\bigr)
=
A
\]
coefficientwise with respect to any $\mathbf{Z}$-basis of $A$,
and the asserted identity follows by applying this equality to
both $x$ and $x^{-1}$. Taking $A=\mathbf{Z}[G_K]_-$ gives
\[
y_K\in\bigl(\mathbf{Z}[G_K]_-\bigr)^\times,
\]
as desired.
\end{proof}

Now we go back and prove that Theorem~\ref{t:ETNCp} implies the freeness of $V_{S'}^T(K)_-$.  

\begin{lemma}  \label{l:free}
Let $G$ be a finite abelian group and $V$ a projective $\Z[G]$-module of constant rank $n \ge 1$.
  Suppose we are given a free $\Z[G]$-module $B$ of rank $n$ and a $\Z[G]$-module map
\[ g \colon V \longrightarrow B. \]
Let $R$ be a quotient of $\Z[G]$, and write $V_R = V \otimes_{\Z[G]} R$, etc.
Suppose that 
\[ \Fitt_R(\coker(g_R)) = xR \] for a nonzero-divisor $x \in R$.  Then $V_R$ is a free $R$-module of rank $n$.
\end{lemma}

\begin{proof}  Of crucial importance is a result of Swan \cite{swan}*{Theorem 1}, which states that there is a $\Z[G]$-module isomorphism $V \cong F \oplus I$ where $F$ is a free module and $I \subset \Z[G]$ is an ideal of finite index.  Note that Swan's result does not require $G$ to be abelian, but the rest of our proof (in particular, the step where we choose generators for $I$ that are nonzero-divisors) does.

  It is clear that $F$ has rank $n-1$ and $I$ is projective of constant rank $1$.  Since $V_R \cong F_R \oplus I_R$, we need to prove that $I_R$ is free of rank 1 over $R$.   We claim that the ideal $I$ can be generated by finitely many elements that are all nonzero-divisors in $\Z[G]$.  Firstly, $I$ is finitely generated since $\Z[G]$ is Noetherian; let $\{a_1, \dots, a_k \}$ be generators.  Write $m$ for the index $[\Z[G] \colon I]$.  For each $a_i$ we can choose an integer $c_i$ such that $a_i + m c_i$ is a nonzero-divisor.  Then the set $\{m, a_1 + m c_1, \dots, a_k + m c_k\}$ consists of nonzero-divisors and generates the ideal $I$.

Let $f_1, \dots, f_{n-1}$ denote a $\Z[G]$-basis of $F$.  Fix also a basis of $B$.  Define a map \[ \psi \colon I \longrightarrow \Z[G], \qquad \psi(z) = \det(g(z), g(f_1), \dots, g(f_{n-1})), \]
where the determinant is with respect to the chosen basis of $B$.  We claim that \begin{equation} \label{e:impsi} \im(\psi) = \Fitt_{\Z[G]}(\coker g). \end{equation}
Indeed, if $d_1, \dots, d_r$ denote generators for the ideal $I$, then $\Fitt_{\Z[G]}(\coker g)$ is by definition the ideal of $\Z[G]$ generated by 
\[  \det(g(d_{i_1}), g(d_{i_2}),  \dots, g(d_{i_\ell}), g(f_{j_1}), g(f_{j_2}),  \dots, g(f_{j_{n-\ell}})), \]
for some choice of indices $1 \le i_1, \dots, i_\ell \le r$, $1 \le j_1, \dots, j_{n - \ell} \le n-1$.  The ideal generated by the terms with $\ell = 1$ is exactly $\im(\psi)$. We claim that any determinants with $\ell > 1$ actually vanish.   To see this, we can assume that the $d_i$ are nonzero-divisors by the claim proven above, and note that 
\begin{align*}
   d_{i_1} d_{i_2} \cdot & \det(g(d_{i_1}), g(d_{i_2}),  \dots, g(d_{i_\ell}), g(f_{j_1}), g(f_{j_2}),  \dots, g(f_{j_{n-\ell}})) \\
  = & \det(d_{i_2} g(d_{i_1}), d_{i_1} g(d_{i_2}),  \dots, g(d_{i_\ell}), g(f_{j_1}), g(f_{j_2}),  \dots, g(f_{j_{n-\ell}})) \\
 = & \det(g(d_{i_1}d_{i_2}),  g(d_{i_1}d_{i_2}),  \dots, g(d_{i_\ell}), g(f_{j_1}), g(f_{j_2}),  \dots, g(f_{j_{n-\ell}})) \\
 = &  0.
  \end{align*}
  The desired vanishing follows.  This proves (\ref{e:impsi}).  Tensoring with $R$, we obtain 
   $\psi_R \colon I_R \longrightarrow R$ such that 
   \[ \im(\psi_R) = \Fitt_{\Z[G]}(\coker g)R = \Fitt_R(\coker g_R) = xR, \]
   where the second equality follows from functorial properties of Fitting ideals.  Since $x$ is a nonzero-divisor in $R$, the ideal $xR$ is a free $R$-module of rank 1.  As $I_R$ is a projective $R$-module of rank 1, it follows that $\psi_R$ induces an isomorphism from $I_R$ to $xR$, and hence that $I_R$ is free of rank 1.
   \end{proof}

We will apply the considerations above with
$V=V_{S'}^T(K)$, $B=B_{S'}(K)$, and $R=\Z[G_K]_-$.
For any $S\subset S'$, we define
\[
g_S^T(K)\colon V_{S'}^T(K)\longrightarrow B_{S'}(K)
\]
using the identification (\ref{e:bisom}) for the target.
At each infinite place $v\neq\infty$, take the same local
component as $f_{S(K)}^T(K)$, independently of $S$.
For finite places $v\in S'$, define the local components
as follows:
\begin{itemize}
\item If $v \in S$, the component of  $g_S^T(K)$ at $v$ is induced by the composition $V_w \rightarrow \Delta G_w \rightarrow \Z[G_w]$ followed by multiplication by $\#I_v$.
\item If $v \in S' \setminus S$, the component of  $g_S^T(K)$ at $v$ is induced by the following.
Suppose $x \in V_w$ maps to $(r,s) \in W_w$, with the model of $W_w$ given by (\ref{e:rww}). 
Map $x$ to $r \cdot (\#I_v - \N I_v) + s \cdot \N I_v$.
\end{itemize}

The map $g$ to which we will apply Lemma~\ref{l:free} is $g_\emptyset^T(K)$, where $\emptyset$ is the empty set.
Define \[ \widetilde{\nabla}_S^T(K) = \coker(g_S^T(K)). \]
Recall the definition of    $S(K)$  in (\ref{e:sdef}).
\begin{theorem} \label{t:fittg}
 Assume Theorem~\ref{t:ETNCp}. Let $m_K = \prod_{v \in S_{\ram}(K/F)} \# I_v$, where the product ranges over the finite primes of $F$ that are ramified in $K$. Then for each $S \subset S(K)$, we have 
\[
 \Fitt_{\Z[G_K]_-}\bigl(\widetilde{\nabla}_S^T(K)_-\bigr)
 =\bigl(m_K\cdot\Theta_{S,T}(K/F)\bigr).
\]
\end{theorem}
Note that for $S = \emptyset$, the element generating the ideal on the right is a nonzero-divisor in $\Z[G_K]_-$.

\begin{proof}
It suffices to prove the result after tensoring with $\Z_p$ for each prime $p$.
Applying Theorem~\ref{t:ETNCp}, fix bases of $V_{S'}^T(K)_{p,-}$ and $B_{S'}(K)_{p,-}$ such that 
\[ \det(f_{S(E)}^T(E)_{p,-}) = \Theta_{S(E), T}(E/F) \qquad \text{ in } \Z_p[G_E]_- \]
for each CM subfield $E \subset K$.

By definition of the maps, we have \begin{equation} \label{e:fittbase}
\det(g_{S(E)}^T(E)_{p,-}) = m_E \cdot \det(f_{S(E)}^T(E)_{p,-}) = m_E \Theta_{S(E), T}(E/F). \end{equation}
We prove simultaneously for every CM intermediate field $E$ of $K/F$
and every $U \subset S(E)$ that
\begin{equation} \label{e:detgs}
\det(g_U^T(E)_{p,-}) = m_E \Theta_{U,T}(E/F).
\end{equation}
We use strong induction on $\#(S(E)\setminus U)$, with the bases induced
from the fixed bases over $K$. The base case is (\ref{e:fittbase}).
Suppose that the equality is known for all such pairs $(E,U)$ with
$\#(S(E)\setminus U)\leq n$. We give the induction step for $K$; the
same argument applies to every CM intermediate field. Choose
$S\subset S(K)$ with $\#(S(K)\setminus S)=n$, let $v\in S$, and put
$S_v=S\setminus\{v\}$.

Let $E_v=K^{I_v}$ and let
\[
 e_v=\frac{1}{\#I_v}\N I_v\in\Q_p[G_K]_-
\]
be the idempotent projecting onto the characters unramified at $v$.
If $E_v$ is totally real, then $c\in I_v$, so
$\N I_v=0$ in $\Z_p[G_K]_-$ and hence $e_v=0$. The local definitions
therefore give
\[
 g_{S_v}^T(K)_{p,-}=g_S^T(K)_{p,-}.
\]
Moreover, every totally odd character is nontrivial on $I_v$, so
\[
 \Theta_{S_v,T}(K/F)=\Theta_{S,T}(K/F).
\]
Thus the desired equality for $S_v$ follows from that for $S$.

We may therefore assume that $E_v$ is CM. In particular, $v$ is finite
and $v\notin S(E_v)$. There is an injection
\begin{equation} \label{e:factorv}
 \Z_p[G_K]_- \longrightarrow
 e_v\Z_p[G_K]_-\times(1-e_v)\Z_p[G_K]_-.
\end{equation}
We prove the desired equality after projection to each factor.
The first factor is naturally identified with $\Z_p[G_{E_v}]_-$,
and projection to it corresponds to $\red_{K/E_v}$.

Put
\[
 A_v=S_v\cap S(E_v),\qquad D_v=S_v\setminus S(E_v).
\]
Since $S(E_v)\subset S(K)$ and $v\notin S(E_v)$, we have
\[
 \#(S(E_v)\setminus A_v)
 =\#(S(E_v)\setminus S_v)
 =\#(S(E_v)\setminus S)
 \leq\#(S(K)\setminus S)=n.
\]
The inductive hypothesis therefore applies to $(E_v,A_v)$.
Every prime $u\in D_v$ is finite and unramified in $E_v/F$.
Let $\sigma_u\in G_{E_v}$ denote its arithmetic Frobenius.
At such a prime, the relation in (\ref{e:rww}) is
$r=(1-\sigma_u^{-1})s$. Consequently, adjoining $u$ to the depletion
set multiplies the corresponding local component of $g$ by
$1-\sigma_u^{-1}$, without changing any other component.
With the fixed bases, it follows that
\begin{align*}
 \det(g_{S_v}^T(E_v)_{p,-})
 &=\left(\prod_{u\in D_v}(1-\sigma_u^{-1})\right)
   \det(g_{A_v}^T(E_v)_{p,-})\\
 &=m_{E_v}\left(\prod_{u\in D_v}(1-\sigma_u^{-1})\right)
   \Theta_{A_v,T}(E_v/F)\\
 &=m_{E_v}\Theta_{S_v,T}(E_v/F).
\end{align*}
The local definitions and their compatibility with co-invariants give
\begin{equation} \label{e:reduce1}
\begin{aligned}
 \red_{K/E_v}\bigl(\det(g_{S_v}^T(K)_{p,-})\bigr)
 &=\frac{m_K}{m_{E_v}}\det(g_{S_v}^T(E_v)_{p,-})\\
 &=m_K\Theta_{S_v,T}(E_v/F).
\end{aligned}
\end{equation}
This is exactly the reduction of $m_K\Theta_{S_v,T}(K/F)$.

Upon projection to the second factor in (\ref{e:factorv}),
$\N I_v$ vanishes, so the maps $g_{S_v}^T(K)_{p,-}$ and
$g_S^T(K)_{p,-}$ become equal. The Euler factor at $v$ acts
as $1$ on this factor, so $\Theta_{S_v,T}(K/F)$ and
$\Theta_{S,T}(K/F)$ also have the same image. The inductive hypothesis
for $(K,S)$ therefore gives the desired equality on this factor.
Injectivity in (\ref{e:factorv}) now proves the equality for $S_v$.
This completes the simultaneous induction and the proof of
(\ref{e:detgs}). The theorem follows.
\end{proof}

Theorem~\ref{t:fittg} for $S = \emptyset$ and Lemma~\ref{l:free} together imply that $V_{S'}^T(K)_-$ is free of rank $\#S' - 1$ over $\Z[G_K]_-$.

\begin{remark}  This proof demonstrates the power of Theorem~\ref{t:ETNCp} and other formulations of $\textrm{ETNC}^{-}$.  By having an equality of elements over $K$ and all its CM subfields simultaneously (with respect to one fixed choice of bases), one can prove results inductively using projections onto factors separately as in (\ref{e:factorv}).  If one is only dealing with equalities of ideals, rather than elements, such arguments tend to break down since the induced map on units in maps such as (\ref{e:factorv}) will in general not be surjective.
\end{remark}

\subsection{Reduction 2: Lemma~\ref{l:alter}} \label{s:red2}

The proof of Lemma~\ref{l:alter} has the same flavor as the computations of \S\ref{s:functoriality}, but is significantly more complicated because of the changing set $S'$.
We should mention that such an ``independence of $(S,T)$ sets'' result was proved by Atsuta and Kataoka in \cite{ak}*{Proposition 3.1(iii)}.  

We begin with an important purely local computation regarding local Ritter--Weiss modules.
Let $L/K$ be an abelian extension of $\ell$-adic local fields (i.e.~finite extensions of $\Q_\ell$).  Write $G = \Gal(L/K)$.  Let $V$ denote the local Ritter--Weiss $G$-module associated to $L$ as in (\ref{e:vseq}).
 Let $U \subset \cO_L^\times$ denote the subgroup of 1-units.  We then obtain a short exact sequence
\begin{equation} \label{e:vdgu}
 1 \lra L^\times/U \lra V/U \lra \Delta G \lra 1. 
 \end{equation}
Let $k_L$ denote the residue field of $L$.

We will work over $\Z_p$ where $p \neq \ell$ is another prime.  For any $G$-module $M$, we write $M_p$ for the $\Z_p[G]$-module $M \otimes_{\Z} \Z_p$.
Let $I \subset G$ be the inertia group.  Let $e = \# I$ be the size of $I$, write $\N I = \sum_{g \in I} g$, and let $\sigma \in G$ be a representative of the Frobenius coset in $G/I$.

\begin{lemma}\label{l:tx}
Put $\mathcal V=(V/U)_p$, and let $q$ be the size of the residue
field of $K$. There exists $\widetilde x\in\mathcal V$ whose image
in the model (\ref{e:rww}) of $W_p$ is
\[
 (x,e),\qquad x=e-\sigma^{-1}\N I.
\]
For $M=\mathcal V/\Z_p[G]\widetilde x$, we have
\begin{equation}\label{e:ydef}
 \Fitt_{\Z_p[G]}(M)=(y),\qquad y=e-\sigma^{-1}q\N I.
\end{equation}
\end{lemma}

\begin{remark}
The proof also shows that every lift $\widetilde x'$ of the same
pair $(x,e)\in W_p$ satisfies
\begin{equation}\label{e:y2}
 \Fitt_{\Z_p[G]}\bigl(\mathcal V/\Z_p[G]\widetilde x'\bigr)=(y).
\end{equation}
Indeed, the difference of two such lifts lies in $(k_L^\times)_p$.
In the cyclic-inertia presentation below, this group is generated
by $\N I\,g_{\widetilde\tau}$, and adding a multiple of this element
to $\widetilde x$ does not change the determinant, since
$(\tau-1)\N I=0$. The same conclusion holds after the
prime-to-$p$ inertia reduction in the proof.
\end{remark}

\begin{proof}
Write $G=G_p\times G'$ and $I=I_p\times I'$, where $G_p$ and $I_p$
are the respective $p$-Sylow subgroups, and put $d=\#I'$.
Decompose
\[
 \Z_p[G]=\prod_{(\chi)}R_{(\chi)},\qquad
 R_{(\chi)}=e_{(\chi)}\Z_p[G],
\]
where $(\chi)$ ranges over the $\Gal(\overline\Q_p/\Q_p)$-orbits
of characters of $G'$ and $e_{(\chi)}=\sum_{\chi\in(\chi)}e_\chi$.
For a $\Z_p[G]$-module $N$, write $N_{(\chi)}=e_{(\chi)}N$.

If $\chi(I')\neq1$, the exact sequence
\[
 1\longrightarrow k_L^\times\longrightarrow L^\times/U
 \longrightarrow\Z\longrightarrow1
\]
shows that $(L^\times/U)_{p,(\chi)}=0$, since inertia acts trivially
on both $k_L^\times$ and $\Z$. Thus
$\mathcal V_{(\chi)}\cong(\Delta G)_{p,(\chi)}\cong R_{(\chi)}$.
Choose the unique lift of the $(\chi)$-component of $x$.
The second coordinate of $W_{p,(\chi)}$ is zero, and
$e_{(\chi)}x=e_{(\chi)}y=e\,e_{(\chi)}$, so both assertions hold
on this component.

For the remaining components, put $L'=L^{I'}$, and let $V(L')$
and $U(L')$ denote the local Ritter--Weiss module and the group
of $1$-units for $L'/K$. View these as $G$-modules by inflation.
Restriction of Weil groups induces an isomorphism
\[
 \mathcal V_{(\chi)}\longrightarrow
 \bigl((V(L')/U(L'))_p\bigr)_{(\chi)}.
\]
To see this, apply the $(\chi)$-component to (\ref{e:vdgu}) for
$L$ and $L'$. The map on the kernels is induced by the norm
$L^\times\to L'^\times$ and is an isomorphism on their $p$-parts
after taking $I'$-co-invariants, since $d$ is prime to $p$.
The map on the augmentation ideals is an isomorphism on these
components as well. Since $L/L'$ is totally ramified, the map
on the models of $W$ is
\[
 (r,s)\longmapsto(\overline r,s),
\]
under the identification $G/I=\Gal(L'/K)/I_p$.
The elements $x$ and $y$ reduce to $d x'$ and $d y'$, where
$x'$ and $y'$ are the corresponding elements for $L'/K$;
also $e=d\#I_p$.
Consequently, a lift for $L'/K$ with image $(x',\#I_p)$,
multiplied by $d$ and pulled back under this isomorphism,
has the required image and Fitting ideal. The factor $d$ is a
unit in $\Z_p$. It therefore suffices to treat the case that
$I$ is a $p$-group. The componentwise choices then combine to an
element of $\mathcal V$; no integral approximation is needed.

Assume henceforth that $I$ is a $p$-group. Since $p\neq\ell$,
$I$ is tame and cyclic. Let $\tau$ generate $I$.
Recall the left-module construction of \cite{rw}*{Proposition 1}.
For the Weil-group extension
\[
 1\longrightarrow L^\times\longrightarrow
 \text{W}(L^{\ab}/K)\longrightarrow G\longrightarrow1,
\]
let $\Delta$ be the augmentation ideal of its integral group ring,
and let $\Delta'$ be the kernel of the map of group rings to
$\Z[G]$. Then
\[
 V=\Delta/(\Delta'\cdot\Delta),
\]
with $G$ acting by left multiplication by lifts.

Choose a uniformizer $\pi_K$ of $K$, and let $\widetilde L/L$
be the abelian extension corresponding by local class field theory
to $L^\times/(U\langle\pi_K\rangle)$. This extension is Galois
over $K$, since the defining subgroup is $G$-stable. Its inertia
group $\widetilde I$ is tame and cyclic, and fits into
\[
 1\longrightarrow k_L^\times\longrightarrow\widetilde I
 \longrightarrow I\longrightarrow1.
\]
Choose a generator of $\widetilde I$ mapping to $\tau$, and choose
$\widetilde\tau$ in the inertia of $\text{W}(L^{\ab}/K)$ mapping
to that generator. Choose also a Frobenius lift
$\widetilde\sigma$ mapping to $\sigma$.
Let $g_{\widetilde\sigma},g_{\widetilde\tau}\in\mathcal V$ be the
images of $1-\widetilde\sigma^{-1}$ and $\widetilde\tau-1$.

The exact sequence (\ref{e:wseq}) gives
\begin{equation}\label{e:vdgw}
 0\longrightarrow(k_L^\times)_p\longrightarrow\mathcal V
 \longrightarrow W_p\longrightarrow0.
\end{equation}
The two elements map to
\[
 g_\sigma=(1-\sigma^{-1},1),\qquad g_\tau=(\tau-1,0).
\]
These generate $W_p$, with relations generated by
\begin{align}
 (\tau-1)g_\sigma-(1-\sigma^{-1})g_\tau&=0,
 \label{e:rel1}\\
 \N I\,g_\tau&=0.\label{e:rel2}
\end{align}
By the choice of $\widetilde\tau$, the element
$\N I\,g_{\widetilde\tau}=\widetilde\tau^e-1$ generates
$(k_L^\times)_p$. Hence the two chosen lifts generate $\mathcal V$.
The annihilator of $(k_L^\times)_p$ is generated by $\tau-1$ and
$1-\sigma^{-1}q$. Lifting (\ref{e:rel2}) therefore gives
\begin{equation}\label{e:dudrel}
 (1-\sigma^{-1}q)\N I\,g_{\widetilde\tau}=0.
\end{equation}
To lift (\ref{e:rel1}), use the tame Frobenius relation modulo
$U$. Since $G$ is abelian, $e$ divides $q-1$, and we obtain
\begin{align*}
 (\tau-1)g_{\widetilde\sigma}
 -(1-\sigma^{-1})g_{\widetilde\tau}
 &=\tau\sigma^{-1}
   (\widetilde\sigma\widetilde\tau^{-1}
    \widetilde\sigma^{-1}\widetilde\tau-1)\\
 &=\tau\sigma^{-1}(\widetilde\tau^{1-q}-1)\\
 &=\sigma^{-1}\frac{1-q}{e}\N I\,g_{\widetilde\tau}.
\end{align*}
Thus the remaining relation is
\begin{equation}\label{e:mainrel}
 (\tau-1)g_{\widetilde\sigma}
 -\left(1-\sigma^{-1}+\sigma^{-1}\frac{1-q}{e}\N I\right)
   g_{\widetilde\tau}=0.
\end{equation}
Multiplying this by $\N I$ gives (\ref{e:dudrel}), so
(\ref{e:mainrel}) alone presents $\mathcal V$ on these generators.

Set
\begin{equation}\label{e:xtdef}
 \widetilde x=e\,g_{\widetilde\sigma}-z\,g_{\widetilde\tau},
\end{equation}
where
\[
 z=\sigma^{-1}\bigl((e-1)+(e-2)\tau+\cdots+\tau^{e-2}\bigr)
\]
if $e>1$, and $z=0$ if $e=1$. Its image in $W_p$ is $(x,e)$,
since $e(1-\sigma^{-1})-z(\tau-1)=x$.
The quotient by $\widetilde x$ has Fitting ideal generated by
\[
 \det\mat{\tau-1}
 {\sigma^{-1}-1-\sigma^{-1}\frac{1-q}{e}\N I}
 {e}{-z}
 =e-\sigma^{-1}q\N I=y.
\]
This proves the lemma.
\end{proof}

We can now prove Lemma~\ref{l:alter}, which states that the $(\Sigma, \Sigma')$ version of the $\textrm{ETNC}^{-}$ over $\Z_p$ implies the $(S,T)$ version over $\Z_p$.

\begin{proof}
The main difficulty in relating the $(\Sigma,\Sigma')$ and $(S,T)$
versions of the $\textrm{ETNC}^{-}$ is the fact that different sets
$S'$ must be taken by the assumptions on this set. Write
\[
 J=S\setminus\Sigma=\Sigma'\setminus T
   =\{v\in S_{\ram}(K/F):v\nmid p\}.
\]
When working with $(\Sigma,\Sigma')$, we start with a set $S'$ that
is necessarily disjoint from $J$. Let us fix this choice of $S'$.
When we work with $(S,T)$, the set that plays the role of $S'$ must
contain $J$, so we use $S'\cup J$. From the definitions\footnote{The
snake lemma must also be used, since the module
$V_{S'\cup J}^{T}(K)$ is the kernel of $\theta_V$.}
one finds that there is a short exact sequence
\begin{equation} \label{e:vuseq}
\begin{tikzcd}
0\ar[r]&V_{S'}^{\Sigma'}(K)\ar[r,"\iota"]
&V_{S'\cup J}^{T}(K)\ar[r]
&\displaystyle\prod_{v\in J}^{\sim}V_w/U_w^1\ar[r]&0.
\end{tikzcd}
\end{equation}
Here $w$ denotes the place of $K$ above the place $v$ of $F$ used
in the definition of the $V$ modules. Meanwhile, the isomorphism
(\ref{e:bisom}) yields
\[
 B_{S'\cup J}(K)_p
 \cong\prod_{v\in(S'\cup J)\setminus\{\infty\}}\Z_p[G_K]
 \cong B_{S'}(K)_p\oplus\prod_{v\in J}\Z_p[G_K].
\]
This identification yields natural bases for $B_{S'}(K)_p$ and
$B_{S'\cup J}(K)_p$.
For each $v\in J$, let $(x_v,\widetilde{x}_v)$ be the pair
$(x,\widetilde{x})$ for the local extension $K_w/F_v$ given in
Lemma~\ref{l:tx}. Choose $X_v\in V_{S'\cup J}^{T}(K)_p$ lifting
$1\otimes\widetilde{x}_v\in
(\Ind_{G_w}^{G_K}(V_w/U_w^1))_p$,
regarded as an element of the $v$-summand with all other
components zero, under the $p$-completion of the quotient map
in (\ref{e:vuseq}).
We consider the following commutative diagram of free
$\Z_p[G_K]$-modules of rank $\#(S'\cup J)-1$:
\begin{equation}\label{e:vbsquare}
\begin{tikzcd}[row sep=large,column sep=huge]
V_{S'}^{\Sigma'}(K)_p\oplus\prod_{v\in J}\Z_p[G_K]
\ar[r,"(f_{\Sigma}^{\Sigma'}(K)_p{,}1)"]
\ar[hookrightarrow]{d}{(\iota{,}X_v)}
&B_{S'}(K)_p\oplus\prod_{v\in J}\Z_p[G_K]
\ar[d,"{\mat{1}{*}{0}{\operatorname{diag}(x_v)}}"]\\
V_{S'\cup J}^{T}(K)_p\ar[r,"f_S^T(K)_p"]
&B_{S'\cup J}(K)_p.
\end{tikzcd}
\end{equation}
The upper-right block of the right vertical arrow is given by the $B_{S'}(K)_p$-components
of $f_S^T(K)_p(X_v)$. The determinant of the
right arrow is  $\prod_{v\in J}x_v$.

By Theorem~\ref{t:ETNCa}, we can choose a basis of
$V_{S'}^{\Sigma'}(K)_p$ such that, with the induced bases,
\[
 \det(f_{\Sigma(E)}^{\Sigma'}(E)_{p,-})
 =\Theta_{\Sigma(E),\Sigma'}(E/F)
\]
for every CM subfield $E\subset K$ containing $F$.
Here the required unit change of basis lifts from
$\Z_p[G_K]_-$ to $\Z_p[G_K]$.
The cokernel of the left arrow of (\ref{e:vbsquare}) is
\[
 \prod_{v\in J}^{\sim}
 \bigl((V_w/U_w^1)_p/\Z_p[G_w]\widetilde{x}_v\bigr),
\]
which by Lemma~\ref{l:tx} has Fitting ideal generated by
$\prod_{v\in J}y_v$, where $y_v$ is the element (\ref{e:ydef}).
Since this product is a nonzero-divisor, the cokernel is finite
and the left arrow is injective. We can therefore choose a basis
for $V_{S'\cup J}^{T}(K)_p$ such that the determinant of the left
arrow is exactly $\prod_{v\in J}y_v$. After projection to the
minus side, the commutative diagram (\ref{e:vbsquare}) yields
\[
 \det(f_S^T(K)_{p,-})\prod_{v\in J}y_v
 =\Theta_{\Sigma,\Sigma'}(K/F)\prod_{v\in J}x_v.
\]
Since the elements $y_v$ are nonzero-divisors, we may divide,
and we obtain
\[
 \det(f_S^T(K)_{p,-})
 =\Theta_{\Sigma,\Sigma'}(K/F)\prod_{v\in J}(x_v/y_v)
 =\Theta_{S,T}(K/F)
\]
as desired.

The same analysis holds for any CM subfield $E\subset K$
containing $F$. Take $\Gal(K/E)$-coinvariants of each term in
(\ref{e:vbsquare}). Replace the horizontal arrows by
$(f_{\Sigma(E)}^{\Sigma'}(E)_p,1)$ and $f_{S(E)}^T(E)_p$.
Leave the left vertical arrow the same. To maintain commutativity,
adjust the upper-right block of the right arrow and replace the
image of $x_v$ by $e_v=\#I_v$ at the places $v\in J$ such that
$v\notin S(E)$, where $I_v$ is the inertia group for $K/F$.
Indeed, Lemma~\ref{l:tx} specifies that $\widetilde{x}_v$ has
image $(x_v,e_v)$ in the local module $W_p$, and the local
component at an unramified place is given by the second
coordinate. The upper-right block again has no effect on the
determinant.

Now, the elements $x_v$ and $y_v$ for $K$ have images in
$\Z_p[G_E]$ equal to the analogous constants for $E/F$ scaled
by the ramification index at $v$ of $K/E$. These extra scaling
factors cancel. Writing $x_v,y_v$ also for these images, and
retaining $e_v=\#I_v$ for $K/F$, we obtain
\begin{align*}
 \det(f_{S(E)}^T(E)_{p,-})
 &=\Theta_{\Sigma(E),\Sigma'}(E/F)
   \prod_{v\in S(E)\setminus\Sigma(E)}(x_v/y_v)
   \prod_{v\in J\setminus S(E)}(e_v/y_v)\\
 &=\Theta_{S(E),T}(E/F).
\end{align*}
This completes the proof.
\end{proof}

We conclude by noting that Lemmas~\ref{l:pbyp} and~\ref{l:alter} together show that Theorem~\ref{t:ETNCa} implies Theorem~\ref{t:ETNC}.  

\section{Consequences}  \label{s:con}

The implications from the LTC to the Rubin--Stark conjecture
and to the refined class number formula of Mazur--Rubin and Sano
are proved in \cite{bks}*{Theorems 5.12 and 5.16}, respectively.
Applied on the minus side and specialized to the relevant data,
these arguments give the consequences stated in \S\ref{s:icon}.
In this section we indicate how Theorem~\ref{t:ETNC}, combined
with the refinements of \cite{bsapet} at the prime $p=2$, yields
these consequences with an improved power of $2$.

\subsection{The Brumer--Stark conjecture (with an improved power of 2)}  \label{s:con1}

Here we adapt the beautiful argument of Burns, Kurihara, and Sano, who proved that their LTC implies all of these results \cite{bks} (without the improved power of 2).
Since our argument is essentially identical to theirs, we provide a sketch and discuss only the Brumer--Stark conjecture, focusing on the issues related to projecting to the minus side and the improved power of 2.

We consider our setup as above, with a CM abelian extension $K$ of a totally real field $F$.   We let $S = S_\infty \cup S_{\ram}(K/F)$.  Let $\fp \not \in S$ denote a prime of $F$ that splits completely in $K$, and suppose that the auxiliary set $S'$ contains $\fp$.
Let us fix one element $v_0 \in S \subset S'$ as follows.  If $K/F$ is ramified at some finite place, we let $v_0$ be such a place; otherwise, we let $v_0$  be a real place of $F$.
Then as in (\ref{e:bisom}) we have
\begin{equation} \label{e:bisom2}
 B_{S'}(K)_- \cong \prod_{v \in S' \setminus \{v_0\}} \Z[G_K]_-, \end{equation}
and we fix the basis of $ B_{S'}(K)_-$ associated to this isomorphism.

Theorem~\ref{t:ETNC} implies that we may choose a basis for $V_{S'}^T(K)_-$ such that if $A$ is the matrix for $f_{S}^{T}(K)_-$ with respect to these bases, then
\[   \det(A) = \Theta_{S,T}(K/F) \quad \text{ in } \Z[G_K]_-. \]
Note that $A$ is a square matrix of dimension $\#S' - 1$ and coefficients in $\Z[G_K]_-$.  
As indicated in (\ref{e:bisom2}), the columns of $A$ are indexed by the elements $v \in S' \setminus \{v_0\}$.

Now let $A_V$ be the matrix $A$ modified so that in the column indexed by $\fp$, the element in each row is replaced by the basis vector in $V_{S'}^T(K)_-$ indexing that row.  So the matrix $A_V$ has all its coefficients in $ \Z[G_K]_-$ except one column, which has coefficients in $V_{S'}^T(K)_-$.  It makes sense to consider $\det(A_V) \in V_{S'}^T(K)_-$ using the Laplace expansion. 

Write $n = [F \colon \Q]$ and let \[ t = \begin{cases} n & \text{if } S_{\ram}(K/F) \text{ is nonempty,} \\
n - 1 & \text{if } S_{\ram}(K/F) \text{ is empty}. \end{cases} \]
Then $S' \setminus \{v_0\}$ contains $t$ real places.  The trick observed in \cite{bsapet} is that all the columns of $A$ indexed by these real places are divisible by 2 (they are divisible by $1-c$ since this element generates the ideal $\Delta G_w$, and $1-c = 2$ in $\Z[G_K]_-$).  Hence we can divide each of these columns by 2 and obtain 
$\det(A_V)/2^t \in V_{S'}^T(K)_-$.  We then define
\[ u = (1-c) \det(A_V)/2^t \in V_{S'}^T(K), \]
which is a well-defined element of $V_{S'}^T(K)$.  This is the key point at which we have lifted from $\Z[G_K]_-$ to $\Z[G_K]$.

\begin{lemma}  Let $S_\fp = S \cup \{\fp\}$. The element $u$ lies in the kernel of $f_{S_\fp}^T \colon  V_{S'}^T(K) \lra  B_{S'}(K)$.
\end{lemma}

\begin{proof}  For any place $v \in S' \setminus \{v_0\}$, the component of $f_{S_\fp}^T(u)$ at $v$ is $(1-c)/2^t$ times the determinant of the matrix $A_v$ obtained from $A_V$ by replacing the column at $\fp$ with its image under the component of $f_{S_\fp}^T$ at $v$.  If $v \neq \fp$, this is precisely the column of $A_V$ at $v$.  Hence the matrix $A_v$ has two identical columns, namely the columns indexed by $\fp$ and $v$.  It follows that $\det(A_v) =0$.  
If $v = \fp$, the component of $f_{S_\fp}^T(u)$ at $\fp$ lies in $\Delta G_\fp$, which vanishes since $\fp$ splits completely in $K$.  Therefore $A_\fp$ has a column of 0's and hence $\det(A_\fp) =0$.
This proves the lemma.
\end{proof}

Now we can view $u$ as an element of $\cO_{K, S_\fp, T}^\times$, since this is the kernel of $f_{S_\fp}^T$ (see \cite{dk}*{Eqn. (145)}). 

\begin{lemma}  We have $|u|_w = 1$ for all places $w \nmid \fp$, finite or infinite.
\end{lemma}

\begin{proof}
For a finite place $w$ outside $(S_\fp)_K$, the assertion follows
from $u\in\cO_{K,S_\fp,T}^\times$. Now let $w$ lie above a finite
place $v\in S$, initially with $v\neq v_0$.

Choose lifts to $V_{S'}^T(K)$ of the basis of $V_{S'}^T(K)_-$
used to define $A$, and let $\widehat A$ be the square matrix whose
rows are their images under $f_S^T(K)$ in the same target coordinates.
Form $\widehat A_V$ by replacing the column at $\fp$ by these lifts,
and put $d=\det(\widehat A_V)\in V_{S'}^T(K)$.
Then $d$ maps to $\det(A_V)$ on the minus side, and
\[
 2^t u=(1-c)d.
\]
As in \cite{igs}*{Lemma 4.8}, the image of $d$ in
$\Ind_{G_w}^{G_K}W_w(K_w)$ is the determinant obtained by replacing
the column at $\fp$ by its image in this module. The column at $v$
is present, and (\ref{e:rww}) implies that the $2\times2$ minors
formed with these two columns vanish. Hence the image of $d$,
and therefore that of $2^tu$, is zero. Since this induced module
is $\Z$-torsion-free, the image of $u$ is zero as well.
By (\ref{e:wseq}), $|u|_w=1$.

It remains to treat $v=v_0$ when $v_0$ is finite.
Choose a real place $v_1\in S$, keep the source basis fixed,
and repeat the cofactor construction using the target coordinates
$S'\setminus\{v_1\}$ instead of $S'\setminus\{v_0\}$.
The change of target coordinates has determinant $\pm1$.
Moreover, since $\fp$ splits completely, its coefficient in
$\theta_B$ is $1-\sigma_\fp^{-1}=0$. Thus this coordinate change
fixes the $\fp$-coordinate and does not mix it with the other
coordinates. The new unscaled cofactor vector is therefore
$\pm\det(A_V)$.
There are now $n-1$ real columns instead of $n$, so the resulting
element $u'$ is $\pm2u$ in additive module notation, or
$u^{\pm2}$ in multiplicative notation. The same kernel argument
as in the preceding lemma applies to $u'$. In the new coordinates
the column at $v_0$ is present, so the case just proved gives
$|u'|_w=1$ for $w\mid v_0$. Hence $|u|_w=1$ there too.

For $w$ complex, $c(u)=u^{-1}$ implies $|u|_w=1$.
\end{proof}

 Recall that defining $V_{S'}^T(K)$ involved choosing a place of $K$ above each place of $F$.  Let $\fP$ denote the place above $\fp$ chosen in this definition. 
Writing $G=G_K$ for simplicity, define \[ \ord_{G} \colon K^\times \lra \Z[G], \qquad x \mapsto \sum_{\sigma \in G} \ord_\fP(\sigma(x)) \sigma^{-1}.  \]
The following is the Brumer--Stark conjecture, with the usual
appearance of $\Theta_{S,T}$ improved to $\Theta_{S,T}/2^{t-1}$.
\begin{theorem} We have 
\[ \ord_G(u) = \frac{\Theta_{S,T}(K/F)}{2^{t-1}}. \]
\end{theorem}

\begin{proof}
The proof of \cite{igs}*{Lemma 4.9} applies essentially without change. We recall the argument.
Consider the map \[ f_\fp\colon \Ind_{G_\fP}^G V_\fP(K_\fP) \lra \Z[G] \]
induced by \[ V_\fP(K_\fP) \lra W_\fP(K_\fP) \cong \Z[G_\fP] = \Z. \]  
Note that if we start with 
$A_V$ and replace the column indexed by $\fp$ with its image under $f_\fp$,
then we obtain precisely the matrix $A$, since the component of $f_S^T$ at $\fp$ is by definition $f_\fp$.

Now, the composition of $f_\fp$ with  \[ K^\times \lra \Ind_{G_\fP}^G V_\fP(K_\fP) \] is precisely the map $\ord_G$.
It follows that 
\[ \ord_G(u) = \frac{(1-c)}{2^{t}}\det(A) = \frac{\Theta_{S,T}(K/F)}{2^{t-1}}. \]
\end{proof}

The other consequences of Theorem~\ref{t:ETNC} stated in \S\ref{s:icon}, with $\Theta_{S,T}$ improved to ${\Theta_{S,T}}/{2^{t-1}}$, follow similarly by mildly adapting the arguments of \cite{bks}.

\subsection{Connection to LTC} \label{s:ltc}

We briefly recall the Leading Term Conjecture (LTC) of Burns, Kurihara, and Sano \cite{bks}.  
Let the notation be as in the introduction.
The map $f_{S}^{T}(K)$ in (\ref{e:fdef}) yields an exact sequence
\begin{equation} \label{e:ovbn}
 \begin{tikzcd}
 0 \ar[r] & \cO_{K, S, T}^\times \ar[r] & V_{S'}^T(K) \ar[r,"f_S^T(K)"] & B_{S'}(K) \ar[r] & \nabla_S^T(K) \ar[r] & 0.  \end{tikzcd} 
 \end{equation}
 Here
\[
\cO_{K,S,T}^\times
 = \{x\in K^\times :
       |x|_w=1\text{ if }w\notin S_K,\quad
       |x-1|_w<1\text{ if }w\in T_K\}.
\]

Let
\[
 X_{K,S}=\ker\left(\bigoplus_{w \in S_K} \Z \longrightarrow \Z\right),
  \qquad (a_w)_{w \in S_K} \mapsto \sum_w a_w,
\]
where $S_K$ is the set of places of $K$ above $S$.
The Ritter--Weiss module $\nabla_S^T(K)$ sits in a short exact sequence
\[
 \begin{tikzcd}
 0\ar[r]&\Cl_S^T(K)\ar[r]&\nabla_S^T(K)\ar[r]
 &X_{K,S}\ar[r]&0.
 \end{tikzcd}
\]
If a subscript $\R$ denotes tensoring with $\R$, this gives
$\nabla_S^T(K)_\R\cong(X_{K,S})_\R$.
The $\Z[G_K]$-module $V_{S'}^T(K)$ is projective of constant rank
$\#S'-1$, and $B_{S'}(K)$ is free of the same rank.
The sequence (\ref{e:ovbn}) yields a sequence of maps
\[ \begin{tikzcd}[
    ,row sep = 0ex
    ,/tikz/column 1/.append style={anchor=base east}
    ,/tikz/column 2/.append style={anchor=base west}
    ]
\det\nolimits_{\Z[G]} V_{S'}^T(K) \otimes (\det\nolimits_{\Z[G]} B_{S'}(K))^{-1} \ar[r] &   \det\nolimits_{\R[G]} V_{S'}^T(K)_\R \otimes (\det\nolimits_{\R[G]} B_{S'}(K)_\R)^{-1} \\
 \ar[r,"\sim"] &  \det\nolimits_{\R[G]} (\cO_{K, S, T}^\times)_\R \otimes (\det\nolimits_{\R[G]} \nabla_{S}^T(K)_\R)^{-1} \\
 \ar[r,"\sim"] &  \det\nolimits_{\R[G]} (\cO_{K, S, T}^\times)_\R \otimes (\det\nolimits_{\R[G]} (X_{K,S})_\R)^{-1} \\
 \ar[r,"\sim"] &  \R[G]
\end{tikzcd}
\]
whose composition we denote by $\iota_{K,S,T}$.  The last isomorphism above employs the regulator map
\[ \begin{tikzcd}[
    ,row sep = 0ex
    ,/tikz/column 1/.append style={anchor=base east}
    ,/tikz/column 2/.append style={anchor=base west}
    ] \lambda\colon (\cO_{K, S, T}^\times)_\R \ar[r,"\sim"] &  (X_{K,S})_\R, \\
u \arrow[mapsto]{r} & -\sum_{w\in S_K}(\log|u|_w)w.
\end{tikzcd}
\]
The Leading Term Conjecture states that the image of $\iota_{K,S,T}$ is the $\Z[G]$-submodule of $\R[G]$ generated by $\Theta^*_{S,T}(K/F)$.  Here $\Theta^*_{S,T}(K/F)$ is the element of $\R[G]$ giving the leading terms of the $L$-functions of each character at $s=0$, i.e.\ such that
\[ \chi(\Theta^*_{S, T}(K/F)) = L_{S,T}^*(\chi^{-1}, K/F, 0) = \lim_{s \ra 0} L_{S,T}(\chi^{-1}, K/F,s)/s^{r_\chi} \neq 0 \]
for the appropriate integer $r_\chi$.
Suppose that $\chi$ is an odd character, and put
\[
 S_\chi=\{w\in S:\chi|_{G_{K,w}}\neq1\},
\]
where $G_{K,w}\subset G_K$ is the decomposition group at $w$.
Every place in $S\setminus S_\chi$ is finite, and $\chi$ is
unramified with trivial Frobenius there.
Then \begin{equation} \label{e:logv}
  L_{S,T}^*(\chi^{-1}, K/F, 0) = L_{S_\chi, T}(\chi^{-1}, K/F, 0) \prod_{w \in S \setminus S_\chi} (\log \N w). \end{equation}
Moreover, if $E_\chi$ denotes the subfield of $K$ fixed by the kernel of $\chi$, then
\begin{equation}
\label{e:ocv}
  L_{S_\chi, T}(\chi^{-1}, K/F, 0) = L_{S(E_{\chi}), T}(\chi^{-1}, E_{\chi}/F,0) \prod_{w \in S_\chi \setminus S(E_{\chi})} (1 - \chi^{-1}(w)). 
  \end{equation}

In \cite{kurihara}*{Proposition 3.5}, Kurihara proves that the minus part of the LTC implies Theorem~\ref{t:ETNC}.  As he states there, the converse is also true by the same reasoning, and we briefly sketch some details.

Let $\iota_{K,S,T,-}$ denote the map obtained from
$\iota_{K,S,T}$ by base change to the minus quotient. Its source is
\[
 \det\nolimits_{\Z[G_K]_-}V_{S'}^T(K)_-
 \otimes
 \bigl(\det\nolimits_{\Z[G_K]_-}B_{S'}(K)_-\bigr)^{-1},
\]
and its target is $\R[G_K]_-$.
Use the analogous notation for each CM intermediate field $E$.
Suppose that bases have been chosen for $V_{S'}^T(K)_-$ and
$B_{S'}(K)_-$ so that their induced bases satisfy
\[
 \det(f_{E,-})=\Theta_{S(E),T}(E/F)
 \quad\text{in }\Z[G_E]_-
\]
for every CM field $E$ with $F\subset E\subset K$.
Let $v$ and $b$ be the corresponding generators of the two
determinant lines. It suffices to prove
\[
 \iota_{K,S,T,-}(v\otimes b)=\Theta^*_{S,T}(K/F)_-
 \quad\text{in }\R[G_K]_-.
\]
This can be checked on the odd characters.
Fix such a character $\chi$ and put $E=E_\chi$.
The comparison in \cite{kurihara}*{Proposition 3.5} gives
\begin{equation}\label{e:kv}
\begin{aligned}
 \chi\bigl(\iota_{K,S,T,-}(v\otimes b)\bigr)
 &=\chi\bigl(\iota_{E,S(E),T,-}
                  (\overline v\otimes\overline b)\bigr)\\
 &\quad\cdot\prod_{w\in S\setminus S_\chi}\log\N w
       \prod_{w\in S_\chi\setminus S(E)}(1-\chi^{-1}(w)),
\end{aligned}
\end{equation}
where $\overline v$ and $\overline b$ are the determinant-line
generators obtained from the induced bases over $E$.
The $\chi$-component of $f_{S(E)}^T(E)_-$ is an isomorphism, so
\[
 \chi\bigl(\iota_{E,S(E),T,-}
                (\overline v\otimes\overline b)\bigr)
 =\chi\bigl(\det(f_{S(E)}^T(E)_-)\bigr)
 =L_{S(E),T}(\chi^{-1},E/F,0).
\]
Combining (\ref{e:logv}), (\ref{e:ocv}), and (\ref{e:kv}) yields
\[
 \chi\bigl(\iota_{K,S,T,-}(v\otimes b)\bigr)
 =L_{S,T}^*(\chi^{-1},K/F,0).
\]
This proves the desired equality in $\R[G_K]_-$.

   \end{document}